\documentclass[a4paper, oneside,12pt]{amsart}

\usepackage[english]{babel}
\usepackage{mathrsfs,amssymb}
\usepackage{mathtools}
\usepackage[colorlinks, citecolor = blue]{hyperref}
\usepackage{tcolorbox}
\usepackage[shortlabels]{enumitem}
\setlist[itemize]{leftmargin=25pt}
\setlist[enumerate]{leftmargin=25pt}

\makeatletter
\newcommand{\leqnomode}{\tagsleft@true}
\newcommand{\reqnomode}{\tagsleft@false}
\makeatother

\usepackage{todonotes}

\newtheorem{theorem}{Theorem}[section]
\newtheorem{lemma}[theorem]{Lemma}
\newtheorem{prop}[theorem]{Proposition}
\newtheorem{cor}[theorem]{Corollary}

\theoremstyle{definition}
\newtheorem{definition}[theorem]{Definition}

\theoremstyle{remark}
\newtheorem{remark}[theorem]{Remark}

\newtheorem{example}[theorem]{Example}

\numberwithin{equation}{section}

\DeclareMathOperator{\osc}{osc}
\DeclareMathOperator{\loc}{loc}

\DeclareMathOperator{\BMO}{BMO}

\newcommand{\N}{\ensuremath{\mathbb{N}}}

\newcommand{\R}{\ensuremath{\mathbb{R}}}
\newcommand{\C}{\ensuremath{\mathbb{C}}}

\newcommand{\mc}{\mathcal}
\newcommand{\ms}{\mathscr}

\DeclarePairedDelimiter\abs{\lvert}{\rvert}

\DeclarePairedDelimiter\cbrace\{\}
\DeclarePairedDelimiter\ha()
\DeclarePairedDelimiter{\ip}\langle\rangle
\DeclarePairedDelimiter{\nrm}\lVert\rVert

\newcommand{\nrmb}[1]{\bigl\|#1\bigr\|}
\newcommand{\absb}[1]{\bigl|#1\bigr|}
\newcommand{\hab}[1]{\bigl(#1\bigr)}
\newcommand{\cbraceb}[1]{\bigl\{#1\bigr\}}
\newcommand{\ipb}[1]{\bigl\langle#1\bigr\rangle}

\newcommand{\nrms}[1]{\Bigl\|#1\Bigr\|}

\newcommand{\has}[1]{\Bigl(#1\Bigr)}

\newcommand{\dd}{\hspace{2pt}\mathrm{d}}

\let \la=\lambda
\let \e=\varepsilon
\let \d=\delta
\let \o=\omega
\let \a=\alpha
\let \f=\varphi
\let \b=\beta

\let \g=\gamma
\let \O=\Omega
\let \si=\sigma

\let \ga=\gamma

\def\avint_#1{\mathchoice{\mathop{\kern 0.2em\vrule width 0.6em height 0.69678ex depth -0.58065ex \kern -0.8em \intop}\nolimits_{\kern -0.4em#1}}{\mathop{\kern 0.1em\vrule width 0.5em height 0.69678ex depth -0.60387ex \kern -0.6em \intop}\nolimits_{#1}} {\mathop{\kern 0.1em\vrule width 0.5em height 0.69678ex depth -0.60387ex \kern -0.6em \intop}\nolimits_{#1}} {\mathop{\kern 0.1em\vrule width 0.5em height 0.69678ex depth -0.60387ex \kern -0.6em \intop}\nolimits_{#1}}}

\allowdisplaybreaks

\begin{document}
\title[Bloom weighted bounds for sparse forms]
{Bloom weighted bounds for sparse forms associated to commutators}

\author[A.K. Lerner]{Andrei K. Lerner}
\address[A.K. Lerner]{Department of Mathematics,
Bar-Ilan University, 5290002 Ramat Gan, Israel}
\email{lernera@math.biu.ac.il}

\author[E. Lorist]{Emiel Lorist}
\address[E. Lorist]{Delft Institute of Applied Mathematics\\
Delft University of Technology \\ P.O. Box 5031\\ 2600 GA Delft\\The
Netherlands}
\email{e.lorist@tudelft.nl}

\author[S. Ombrosi]{Sheldy Ombrosi}
\address[S. Ombrosi]{Departamento de Análisis Matemático y Matemática Aplicada\\ Universidad Complutense (Spain) \&
Departamento de Matemática e Instituto de Matemática. Universidad Nacional del Sur - CONICET Argentina}

\email{sombrosi@ucm.es}

\thanks{The first author was supported by ISF grant no. 1035/21. The third author was supported by FONCyT PICT 2018-02501 \& by Spanish Government Ministry
of Science grant PID2020-113048GB-I00.}

\begin{abstract}
In this paper we consider bilinear sparse forms intimately related to iterated commutators of a rather general class of operators.
We establish Bloom weighted estimates for these forms in the full range of exponents, both in the diagonal and off-diagonal cases.
As an application, we obtain new Bloom bounds for commutators of (maximal) rough homogeneous singular integrals and the Bochner-Riesz operator
at the critical index.

We also raise the question about the sharpness of our estimates. In particular we obtain the surprising fact that even in the case of Calder\'on--Zygmund operators,
the previously known quantitative Bloom weighted estimates are not sharp for the second and higher order commutators.
\end{abstract}




\maketitle
\section{Introduction}
Let ${\mathcal S}$ be a sparse family of dyadic cubes, let $b\in L^1_{\text{loc}}({\mathbb R}^n)$, $m\in {\mathbb N}$ and $1\le r<s\le \infty$.
The key object of this paper is the bilinear sparse form defined by
$${\mathcal B}_{{\mathcal S},b,r,s}^m(f,g):=\sum_{Q\in {\mathcal S}}\big\langle |b-\langle b\rangle_Q|^{m}|f|\big\rangle_{r,Q}\langle |g|\rangle_{s',Q}|Q|.$$
This object appears naturally when one studies iterated commutators of various operators $T$ and pointwise multiplication by a function $b$.

Let $r<p,q<s$ and let $\mu, \lambda$ be weights. Our goal is to obtain quantitative weighted $L^p(\mu)\times L^{q'}(\la^{1-q'})$-bounds for ${\mathcal B}_{{\mathcal S},b,r,s}^m$ in the Bloom setting \cite{Blo85} both in the diagonal
and off-diagonal cases. By the Bloom setting one means that an assumption on $b$ is imposed in terms of the Bloom weight $\nu$ depending on $\mu$ and $\la$ in a suitable way.

In the following cases the Bloom bounds for ${\mathcal B}_{{\mathcal S},b,r,s}^m$ have been considered before:
\begin{itemize}
\item $r=1$, $s=\infty$, $m\ge 1$ and $p=q$ \cite{LOR17, LOR19}.
\item $r>1$, $s=\infty$, $m\ge 1$ and $p=q$ \cite{Li22}.
\item $r=1$, $s=\infty$, $m=1$ and $p\neq q$  \cite{HLS23,HOS23}.
\end{itemize}
Our results below are quantitative and cover all possible combinations of $1\le r<p,q<s\le\infty$ and $m\ge 1$. In particular, the bounds we obtain are new in the following settings:
\begin{itemize}
\item Limited range: $r>1$ or $s<\infty$.
\item Iterated and off-diagonal: $m>1$ and  $p\not=q$.
\end{itemize}

In the Bloom setting, prior works have been primarily focused on estimates for commutators of Calder\'on--Zygmund operators, which by \cite{LOR17, LOR19} boil down to estimates for ${\mathcal B}_{{\mathcal S},b,1,\infty}^m$. The study of the boundedness in the case $m=1$
for these operators in the full range $p,q \in (1,\infty)$ has recently been completed by H\"anninen--Sinko and the second author \cite{HLS23}.
For a comprehensive overview of the development of both unweighted and (Bloom) weighted estimates for these commutators, as well a discussion on the necessity of the conditions on $b$, we direct the reader to the introductions of, e.g.,  \cite{HLS23,Hyt21, HOS23}.

Our key application is a quantitative Bloom weighted estimate for iterated commutators of a rather general class of operators.
This class includes, for example, Calder\'on--Zygmund operators, (maximal) rough homogeneous singular integral operators and Bochner--Riesz operators at the critical index. We refer to \cite[Remark 4.4]{LLO22} for a further list of operators that fall within the scope of our theory. Note that for various operators on this list, no Bloom weighted (or even unweighted) commutator estimates were known previously.

 Our approach will build upon a sparse domination procedure for commutators developed by Rivera-R\'ios, the first and third authors \cite{LOR17}, with subsequent generalizations by various authors. We will proceed in 3 steps:
 \begin{enumerate}[(1)]
 \item\label{it:plan1} In Section \ref{sec:sparsedom}, we will prove that iterated commutators of certain sublinear operators $T$ can be dominated by two sparse forms: ${\mathcal B}_{{\mathcal S},b,r,s}^m(f,g)$ and the dual form ${\mathcal B}_{{\mathcal S},b,s',r'}^m(g,f)$. Our key novel point here is in Lemma \ref{two}, which allows one to reduce $m+1$ sparse forms to only 2 sparse forms.
 \item\label{it:plan2} We prove Bloom weighted estimates for ${\mathcal B}_{{\mathcal S},b,r,s}^m(f,g)$. To do so, we first extend a result of Li \cite{Li17} and Fackler--Hyt\"onen~\cite{FH18} to certain fractional sparse forms in Section \ref{sec:fracsparse}. Afterwards, in Section~\ref{sec:sparseB}  we combine the proof strategy of H\"anninen, Sinko and the second author \cite{HLS23} with the change of measure formula of Cascante--Ortega--Verbitsky \cite{COV04} to estimate ${\mathcal B}_{{\mathcal S},b,r,s}^m(f,g)$ in the Bloom setting.

      Furthermore, in the case $q\leq p$ and $m \geq 2$, we provide a second Bloom weighted estimate for ${\mathcal B}_{{\mathcal S},b,r,s}^m(f,g)$ using a proof strategy suggested by Li \cite{Li22}. In particular, Theorem \ref{pleq} presents two incomparable quantitative bounds based on the approaches from \cite{LOR19} and \cite{Li22}.
 \item Combining the first two steps, in Section \ref{sec:application}, we obtain quantitative Bloom weighted estimates for iterated commutators. We apply this result to Calder\'on--Zygmund operators, (maximal) rough homogeneous singular integral operators and Bochner--Riesz operators at the critical index.
 \end{enumerate}

\medskip
Since our estimates are quantitative, it is natural to ask about their sharpness. Here we encounter with an interesting phenomenon, which is new even for Calder\'on--Zygmund operators.
To be more precise, in~\cite{LOR19}, the first and the third authors jointly with Rivera-R\'ios showed for a Calder\'on--Zygmund operator $T$ and for $m\ge 1$ (extending their previous work \cite{LOR17} for $m=1$) that
\begin{equation}\label{lor19}
\|T_b^m\|_{L^p(\mu)\to L^p(\la)}\lesssim \|b\|_{{\BMO}_{\nu}}^m([\la]_{A_p}[\mu]_{A_p})^{\frac{m+1}{2}\max(1,\frac{1}{p-1})},
\end{equation}
where $\BMO_{\nu}$ stands for the weighed $\BMO$ space with weight $\nu:=(\mu/\la)^{1/pm}$, and $T_b^m$ is the $m$-th order commutator of $T$ with a locally integrable function~$b$.

Intuitively, one could conjecture that (\ref{lor19}) is sharp, since in the case of equal weights
$\la=\mu=w$ we obtain the sharp one-weight estimate proved in \cite{CPP12} by Chung, Pereyra and P\'erez. However, following this intuition, any bound by $[\la]_{A_p}^{\a_p}[\mu]_{A_p}^{\b_p}$ with $\a_p$ and $\b_p$ satisfying $$\a_p+\b_p=(m+1)\max(1,\tfrac{1}{p-1})$$ would be sharp.

Observe that the notion of {\it sharpness} in the Bloom setting (or in
the two-weight setting, in general) has not been defined before. It is easy to see that a bound by $[\la]_{A_p}^{\a_p}[\mu]_{A_p}^{\b_p}$ is stronger than a bound by $[\la]_{A_p}^{\a'_p}[\mu]_{A_p}^{\b'_p}$
if and only if $\a_p\le \a'_p$ and $\b_p\le \b'_p$ and at least one of these inequalities is strict. In the case if, for example, $\a_p<\a'_p$ and $\b_p>\b'_p$,
the bounds will be incomparable. This leads us to the following definition.

\begin{definition}\label{sharp}
Let $p\in (1,\infty)$, $\mu,\lambda \in A_p$ and let $T$ be an operator. We say that the estimate
$$
\|T\|_{L^p(\mu)\to L^p(\la)}\lesssim [\la]_{A_p}^{\a_p}[\mu]_{A_p}^{\b_p}
$$
is sharp if neither of the exponents $\a_p$ and $\b_p$ can be decreased.
\end{definition}

Having this definition at hand, we are ready to present our result about the sharpness of (\ref{lor19}). This result comes as a surprise to us because it says that
the estimate (\ref{lor19}) is sharp for all $1<p<\infty$ only if $m=1$. To be more precise, we have the following.

\begin{theorem}\label{shm1}
Let $T$ be a Dini-continuous Calder\'on--Zygmund operator.
\begin{enumerate}[(i)]
\item If $m=1$, then the estimate (\ref{lor19}) is sharp for all $p \in (1,\infty)$.
\item If $m\ge 2$, the estimate (\ref{lor19}) is not
sharp for all $p\not\in [\frac{1+3m}{2m},\frac{1+3m}{m+1}]$.
\end{enumerate}
\end{theorem}

In other words, if $m\ge 2$ and $p\not\in [\frac{1+3m}{2m},\frac{1+3m}{m+1}]$, the second Bloom weighted estimate obtained in Section \ref{sec:sparseB} is incomparable with (\ref{lor19}), and therefore, combined with the result from  \cite{CPP12}, a sharp bound for $T_b^m$ in the
sense of Definition \ref{sharp} does not exist. Observe that Theorem \ref{shm1} leaves open an interesting question about the sharpness of (\ref{lor19}) when $m\ge 2$ and $p\in [\frac{1+3m}{2m},\frac{1+3m}{m+1}]$.

We shall see in Section 6 that a similar phenomenon with two incomparable bounds holds for a large class of operators.

\subsection*{Notation}
We will make extensive use of the notation ``$\lesssim$'' to indicate inequalities up to an implicit multiplicative constant. These implicit constants may depend on $p,q,n,m$, but not on any of the functions under consideration. If these implicit constants depend on the weights $\mu, \lambda$, this will be denoted by ``$\lesssim_{\mu,\lambda}$''.

\section{Preliminaries}\label{sec:prelim}
\subsection{Dyadic lattices} Denote by ${\mathcal Q}$ the set of all cubes $Q\subset {\mathbb R}^n$ with sides parallel to the axes. For a cube $Q \in \mc{Q}$ with side length $\ell(Q)$ and $\alpha>0$ we denote the cube with the same center as $Q$ and side length $\alpha\ell(Q)$ by $\alpha Q$.

Given a cube $Q\in {\mathcal Q}$, denote by
${\mathcal D}(Q)$ the set of all dyadic cubes with respect to $Q$, that is, the cubes obtained by repeated subdivision of $Q$ and each of its descendants into $2^n$ congruent subcubes.
Following \cite[Definition 2.1]{LN19}, a dyadic lattice ${\mathscr D}$ in ${\mathbb R}^n$ is any collection of cubes such that
\begin{enumerate}
\renewcommand{\labelenumi}{(\roman{enumi})}
\item Any child of $Q\in{\mathscr D}$ is in ${\mathscr D}$ as well, i.e. $\mc{D}(Q) \subseteq \ms{D}$.
\item
Any $Q',Q''\in {\mathscr D}$ have a common ancestor, i.e. there exists a $Q\in{\mathscr D}$ such that $Q',Q''\in {\mathcal D}(Q)$.
\item
For every compact set $K\subset {\mathbb R}^n$, there exists a cube $Q\in {\mathscr D}$ containing $K$.
\end{enumerate}
Throughout the paper, $\ms{D}$ will always denote a dyadic lattice.

\begin{definition} Let $\eta \in (0,1)$ and let ${\mathcal S}\subset \mc{Q}$ be a family of cubes. We say that ${\mathcal S}$ is $\eta$-sparse if, for every cube $Q\in {\mathcal S}$, there exists a subset $E_Q\subseteq Q$ such that
$|E_Q|\ge \eta|Q|$ and the sets $\{E_Q\}_{Q\in {\mathcal S}}$ are pairwise disjoint.
We will omit the sparseness number $\eta$ when its value is non-essential.
\end{definition}

For a cube $Q \in \mc{Q}$ and $f \in L^1_{\loc}(\R^n)$ we define $\ip{f}_Q:= \frac{1}{\abs{Q}} \int_Q f$ and for $r>0$ and a positive function $f \in L^r_{\loc}(\R^n)$ we set
$$
  \ip{f}_{r,Q}:= \ip{f^r}_{Q}^{1/r}= \has{\frac{1}{\abs{Q}} \int_Q f^r}^{1/r}.
$$
We define the maximal operator by
\begin{align*}
  Mf&:= \sup_{Q\in \mc{Q}} \,\ip{\abs{f}}_{1,Q} \chi_Q\intertext{and set}
  M_rf&:=M(\abs{f}^r)^{1/r}= \sup_{Q\in \mc{Q}} \ip{\abs{f}}_{r,Q} \chi_Q.
\end{align*}

\subsection{Weights} By a weight $w$ we mean a non-negative $w \in L^1_{\loc}(\R^n)$. For $1<p<\infty$ we say that $w$ belongs to the Muckenhoupt $A_p$-class and write $w \in A_p$ if
$$
[w]_{A_p}:=\sup_{Q\in {\mathcal Q}} \,\ip{w}_{1,Q} \ip{w^{-1}}_{\frac{1}{p-1},Q}<\infty.
$$
For $1 \leq r<\infty$ we say that $w$ belongs to reverse H\"older class and write $w \in {\rm{RH}_r}$ if
$$[w]_{{\rm RH}_r}:=\sup_{Q\in {\mathcal Q}}\frac{\langle w\rangle_{r,Q}}{\langle w\rangle_{1,Q}}<\infty.$$
Furthermore, we say that $w$ belongs to the Muckenhoupt $A_\infty$-class and write $w \in A_p$ if
$$
[w]_{A_{\infty}}:=\sup_{Q\in {\mathcal Q}}\frac{1}{w(Q)}\int_QM(w\chi_Q).$$

We will frequently use that by the definition of the $A_p$-constant, we have
\begin{equation}\label{dualityAp}
  [w^{1-p'}]_{A_{p'}}=[w]_{A_p}^{\frac{1}{p-1}}.
\end{equation}
Moreover, we have
\begin{equation*}
  [w]_{A_{\infty}}\le c_{n}[w]_{A_p},
\end{equation*}
by \cite[Proposition 2.2]{HP13}.

The following quantitative self-improvement lemma from \cite{HPR12} will play a key role in our applications.
\begin{prop}[{\cite[Theorem 1.1 and 1.2]{HPR12}}]\label{rhweights}
 There exists a constant $c_n>0$ such that
for $w \in A_p$ with $1<p<\infty$ we have
$$[w]_{{\rm{RH}}_{1+\frac{1}{c_n[w]_{A_{p}}}}}\leq c_n
\qquad\text{ and }\qquad [w]_{A_{p-\varepsilon}}\leq c_n \, [w]_{A_p}
$$
with $\varepsilon = \frac{p-1}{1+c_n[w]^{\ha{p-1}^{-1}}_{A_{p}}}$.
\end{prop}
\section{A sparse domination principle for commutators}\label{sec:sparsedom} In this section we will prove a general sparse domination principle for iterated commutators, following the line of research started in \cite{LOR17} by Rivera-R\'ios  and the first and the third authors. In order to state our result, let us introduce some notation.

Given a linear operator $T$ and $b \in L^1_{\loc}(\R^n)$, define the first order commutator $T_b^1$ by
$$T_b^1(f):= bT(f)-T(bf).$$
Next, for $m\in {\mathbb N}, m\ge 2$, define higher order commutators $T_b^m$ inductively by
$$T_b^m(f):=T_b^1(T_b^{m-1}(f)).$$
It is easy to see that
\begin{equation}\label{commformula}
T_b^mf(x)= T\hab{(b(x)-b(\cdot))^mf}(x), \qquad x \in \R^n.
\end{equation}

Assume now that $T$ is a general, not necessarily linear, operator. Then we use formula (\ref{commformula}) as the {\it definition} of $T_b^m$.

For $1\leq s\le \infty$ we define the sharp grand maximal truncation operator
$$
{\mathcal M}^{\#}_{T,s}f(x):=\sup_{Q\ni x}\osc_s\big(T(f\chi_{{\mathbb R}^n\setminus 3Q});Q\big), \qquad x \in \R^n,
$$
where
$$\osc_s(f;Q):=\has{\frac{1}{\abs{Q}^2}\int_{Q\times Q}|f(x')-f(x'')|^s\dd x'\dd x''}^{1/s}$$
and the supremum is taken over all $Q \in \mc{Q}$ containing $x$.

We will use the following boundedness property of $T$ and $\mc{M}_{T,s}^{\#}$.

\begin{definition}\label{localweak}
Given an operator $T$ and $r \in [1,\infty)$, we say that $T$ is \emph{locally weak $L^r$-bounded} if there exists a non-increasing function $\f_{T,r}\colon (0,1) \to [0,\infty)$ such that
for any cube $Q\in \mc{Q}$ and $f \in L^r(Q)$ one has
$$
\absb{\cbraceb{x \in Q:|T(f \chi_Q)(x)|>\varphi_{T,r}(\lambda)\ip{|f|}_{r,Q}}} \leq \lambda\, \abs{Q}, \qquad \lambda \in (0,1).
$$
\end{definition}

This definition was given in \cite{LO20} and was called the $W_r$ property of~$T$. Note that the usual weak $L^r$-boundedness of $T$ implies the local weak $L^r$-boundedness of $T$ with
$$\varphi_{T,r}(\lambda) := \lambda^{-1/r}\nrm{T}_{L^r(\R^n) \to L^{r,\infty}(\R^n)}, \qquad \lambda \in (0,1).$$
Moreover, if $T$ is locally weak $L^{r_0}$-bounded for some $r_0 \in [1,\infty)$, it is locally weak $L^r$-bounded for all $r>r_0$ by H\"older's inequality
with $\varphi_{T,r}(\lambda)=\varphi_{T,r_0}(\lambda)$.

The main result of this section is the following abstract sparse domination principle for iterated commutators.

\begin{theorem}\label{sdp} Let $1\le r<s\le\infty, m\in {\mathbb N}$ and let $T$ be a sublinear operator.
Assume that $T$ and ${\mathcal M}^{\#}_{T,s}$ are locally weak $L^r$-bounded. Then there exist $C_{m,n}>1$ and $\la_{m,n}<1$ so that,
for any $f,g\in L^{\infty}_c({\mathbb R}^n)$ and $b\in L^1_{\text{loc}}({\mathbb R}^n)$, there is a $\frac{1}{2\cdot 3^n}$-sparse collection of cubes
${\mathcal S}$ such that
\begin{align*}
\int_{{\mathbb R}^n}|T_b^mf||g|&\le C\Big(\sum_{Q\in {\mathcal S}}\big\langle |b-\langle b\rangle_Q|^{m}|f|\big\rangle_{r,Q}\big\langle|g|\big\rangle_{s',Q}|Q|\\
&\hspace{1cm}+\sum_{Q\in {\mathcal S}}\big\langle |f|\big\rangle_{r,Q}\big\langle |b-\langle b\rangle_Q|^{m}|g|\big\rangle_{s',Q}|Q|\Big),
\end{align*}
where
\begin{equation}\label{constant}
C:=C_{m,n}\big(\f_{T,r}(\la_{m,n})+\f_{\mc{M}^{\#}_{T,s},r}(\la_{m,n})\big).
\end{equation}
\end{theorem}

We refer to \cite[Remark 4.4]{LLO22} for a list of operators satisfying the assumptions of Theorem \ref{sdp}.
Theorem \ref{sdp} is an immediate corollary of the following two statements.

\begin{theorem}\label{sdp1} Under the assumptions of Theorem \ref{sdp} we have
$$\int_{{\mathbb R}^n}|T_b^mf||g|\le C\sum_{k=0}^m\Big(\sum_{Q\in {\mathcal S}}\big\langle |b-\langle b\rangle_Q|^{m-k}|f|\big\rangle_{r,Q}
\big\langle |b-\langle b\rangle_Q|^{k}|g|\big\rangle_{s',Q}|Q|\Big),
$$
where $C$ is given by \eqref{constant}
\end{theorem}

\begin{lemma}\label{two} Let $1\le r,t<\infty$ and $m\in {\mathbb N}$. Let $f,g\in L^{\infty}_c({\mathbb R}^n)$ and $b\in L^1_{\text{loc}}({\mathbb R}^n)$. Fix a cube $Q \in \mc{Q}$ and for $0\le k\le m$ define
$$
c_k:=\big\langle |b-\langle b\rangle_Q|^{m-k}|f|\big\rangle_{r,Q}\big\langle |b-\langle b\rangle_Q|^{k}|g|\big\rangle_{t,Q}.
$$
Then we have $c_k\le c_0+c_m.$
\end{lemma}

Indeed, note that Lemma \ref{two} allows us to reduce the summation over $k=0,\dots,m$ in Theorem \ref{sdp1} to the two extreme terms $k=0$ and $k=m$, yielding the formulation of Theorem \ref{sdp}.

\medskip

Before turning to the proofs, let us mention a brief history of the above results.

\begin{itemize}
\item
In the case where $T$ is a Dini-continuous Calder\'on-Zygmund operator, $m=1$, $r=1$ and  $s=\infty$, Theorem \ref{sdp} goes back to Rivera-R\'ios  and the first and the third authors \cite{LOR17}.
\item
In the case where $m\ge 1$ and $T$ is a generalized H\"ormander singular integral operator, the corresponding version of Theorem \ref{sdp1} was obtained by
Iba\~nez-Firnkorn and Rivera-R\'ios \cite{IR20}.
\item
The closest precursors of Theorem \ref{sdp} were obtained by
\begin{itemize}
  \item Rivera-R\'ios \cite[Theorem 3.1]{R18} in the case $m=1$. We note that in this work  the bilinear maximal operator ${\mathcal M}_T(f,g)$, introduced in \cite{Le19},  was used instead of ${\mathcal M}^{\#}_{T,s}$.
  \item Iba\~nez-Firnkorn and Rivera-R\'ios \cite[Theorem 4.4]{IR22} in the case $m\ge 1$ and with ${\mathcal M}^{\#}_{T,\infty}$.
\end{itemize}
\item For a general account of similar sparse domination results we refer to our recent work \cite{LLO22}.
\end{itemize}

Comparing to \cite[Theorem 3.1]{R18} and \cite[Theorem 4.4]{IR22}, our novel points are the following.
\begin{itemize}
\item Instead of ${\mathcal M}_T(f,g)$ or ${\mathcal M}^{\#}_{T,\infty}$, we deal with a more flexible operator ${\mathcal M}^{\#}_{T,s}$.
Here we continue the line of research originated in \cite{LLO22,LO20,L21}, where various variants of ${\mathcal M}^{\#}_{T,s}$ were considered.
\item We use a \emph{local} weak $L^r$-boundedness assumption, originating from \cite{LO20}, rather than the usual weak $L^r$-boundedness assumption on $T$ and $M^{\#}_{T,s}$.
\item Our most important novel point in this section is in Lemma \ref{two}, which seems to be new. This lemma allows us to significantly  simplify the main applications of Theorem \ref{sdp} to quantitative weighted norm inequalities.
\end{itemize}

The proof of Lemma \ref{two} is quite elementary and we therefore present it first.

\begin{proof}[Proof of Lemma \ref{two}] For $x,y \in \R^n$ denote
$$\f(x,y):=|b(x)-\langle b\rangle_Q|^{m-k}|f(x)||b(y)-\langle b\rangle_Q|^{k}|g(y)|\chi_{Q\times Q}(x,y).$$
Then $c_k$ for $0\leq k\leq m$ can be written in the form
$$c_k=\Big\|x \mapsto\|y \mapsto \f(x,y)\|_{L^t(\frac{\dd y}{|Q|})}\Big\|_{L^r(\frac{\dd x}{|Q|})}.$$
From this, we obtain the conclusion by using the estimate
\begin{equation*}
|b(x)-\langle b\rangle_Q|^{m-k}|b(y)-\langle b\rangle_Q|^{k}\le |b(x)-\langle b\rangle_Q|^{m}+|b(y)-\langle b\rangle_Q|^{m}
\end{equation*}
along with Minkowski's inequality.
\end{proof}

Next we turn to the proof of Theorem \ref{sdp1}. Its proof is based on the well-known ideas developed in the previous works (e.g., \cite{Le19,LLO22,LO20,L21}).

\begin{proof}[Proof of Theorem \ref{sdp1}]
Let $Q \in \mc{Q}$ be a cube that contains the supports of $f$ and $g$. We will show that there exists a $\frac{1}{2}$-sparse family ${\mathcal F}\subset\mc{D}(Q)$ such that
\begin{equation}
\begin{aligned}
  &\int_{{\mathbb R}^n}|T_b^m(f)||g|=\int_{Q}|T_b^m(f\chi_{3Q})||g|\label{local}\\
&\hspace{0.1cm}\le C\sum_{k=0}^m\Big(\sum_{P\in {\mathcal F}}\big\langle |b-\langle b\rangle_{3P}|^{m-k}|f|\big\rangle_{r,3P}
\big\langle |b-\langle b\rangle_{3P}|^{k}|g|\big\rangle_{s',3P}|P|\Big),
\end{aligned}
\end{equation}
where $C$ is given by (\ref{constant}). Taking $\mc{S} = \cbrace{3P:P \in \mc{F}}$ afterwards yields the result.

We construct the family ${\mathcal F}\subset\mc{D}(Q)$ inductively. Set ${\mathcal F}_0=\{Q\}$. Next, given a collection of pairwise disjoint cubes ${\mathcal F}_j$, let
us describe how to construct ${\mathcal F}_{j+1}$.

Fix a cube $P\in {\mathcal F}_j$. For $k=0,\dots,m$ denote
$$\eta_k:=(b-\ip{b}_{3P})^kf\chi_{3P},$$
and consider the sets
$$\O_k(P):=\cbraceb{x\in P:|T(\eta_{k})(x)|>\f_{T,r}(\tfrac{1}{(m+1)6^{n+2}})\langle|\eta_k|\rangle_{r,3P}}$$
and
$${\mathcal M}_k(P):=\cbraceb{x\in P:|\mc{M}^{\#}_{T,s}(\eta_{k})(x)|>\f_{\mc{M}^{\#}_{T,s},r} (\tfrac{1}{(m+1)6^{n+2}})\langle|\eta_k|\rangle_{r,3P}}.$$
Then
$$|\O_k(P)|\le \tfrac{1}{(m+1)6^{n+2}}|3P|\leq\tfrac{1}{3(m+1)2^{n+2}}|P|,$$
and the same bound holds for $|{\mathcal M}_k(P)|$.
Since the maximal operator $M_r$ is  weak $L^r$-bounded with constant independent of $r$, there exists a $c_{n,m}>0$ such that
$$M_k(P):=\cbraceb{x\in P:M_r(\eta_{k})(x)>c_{n,m}\langle|\eta_k|\rangle_{r,3P}}$$
also satisfies
$$|M_k(P)|\le \tfrac{1}{3(m+1)2^{n+2}}|P|.$$
Therefore, setting
$$\O(P):=\bigcup_{k=0}^m\hab{\O_k(P)\cup {\mathcal M}_k(P)\cup M_k(P)},$$
we have $|\O(P)|\le \frac{1}{2^{n+2}}|P|$.

We apply the local Calder\'on--Zygmund decomposition to $\chi_{\O(P)}$ at height $\frac{1}{2^{n+1}}$. We obtain a family of
pairwise disjoint cubes $\mc{S}_P\subseteq \mc{D}(P)$ such that
$\abs{\Omega(P)\setminus\bigcup_{P' \in \mc{S}_P}P'}=0$ and for every $P'\in \mc{S}_P$,
\begin{equation}\label{cz}
\tfrac{1}{2^{n+1}}|P'|\le |P'\cap\Omega(P)|\le \tfrac{1}{2}|P'|.
\end{equation}
In particular, it follows  that
\begin{equation}\label{half}
\sum_{P'\in \mc{S}_P}|P'|\le 2^{n+1}|\O(P)|\le \tfrac12 |P|.
\end{equation}
We define  ${\mathcal F}_{j+1}=\cup_{P\in {\mathcal F}_j}\mc{S}_P$.
Setting ${\mathcal F}=\cup_{j=0}^{\infty}{\mathcal F}_j$, we note  by
\eqref{half} that ${\mathcal F}$ is $\frac12$-sparse.

Now, by iteration, to prove (\ref{local}) it suffices to show for $j \in \N$ and $P \in \mc{F}_j$ that
\begin{align*}
   \int_{P} \absb{T_b^m(f\chi_{3P})}|g| &\le C\sum_{k=0}^m\big\langle |b-\langle b\rangle_{3P}|^{m-k}|f|\big\rangle_{r,3P}
\big\langle |b-\langle b\rangle_{3P}|^{k}|g|\big\rangle_{s',3P}|P|
\\&\hspace{1cm}
    +\sum_{P' \in \mc{F}_{j+1}:P' \subseteq P}\int_{P'}\absb{T_b^m(f\chi_{3P'})}|g|,
\end{align*}
where $C$ is given by (\ref{constant}).
Set $F_j:=\cup_{P\in {\mathcal F}_j}P$.  Noting that
  \begin{align*}
    \int_{P} \absb{T_b^m(f\chi_{3P})}|g| &\leq \int_{P\setminus F_{j+1}} \absb{T_b^m(f\chi_{3P})}|g|\\
    &\hspace{1cm} +\sum_{P' \in \mc{F}_{j+1} : P' \subseteq P} \int_{P'} \absb{T_b^m(f\chi_{3P \setminus 3P'})}|g|
    \\&
   \hspace{1cm} +\sum_{P' \in \mc{F}_{j+1}:P' \subseteq P}\int_{P'}\absb{T_b^m(f\chi_{3P'})}|g|,
  \end{align*}
  it thus suffices to show that
  \begin{align}
  \label{eq:toproveend}
  \begin{aligned}
     &\int_{P\setminus F_{j+1}} \absb{T_b^m(f\chi_{3P})}|g| +\sum_{P' \in \mc{F}_{j+1} : P' \subseteq P} \int_{P'} \absb{T_b^m(f\chi_{3P \setminus 3P'})}|g|\\
    &\hspace{1cm}\le C\sum_{k=0}^m\big\langle |b-\langle b\rangle_{3P}|^{m-k}|f|\big\rangle_{r,3P}
\big\langle |b-\langle b\rangle_{3P}|^{k}|g|\big\rangle_{s',3P}|P|.
  \end{aligned}
  \end{align}

   We first consider the first term on the left-hand side of \eqref{eq:toproveend}. Since
 $T_b^mf=T_{b-c}^mf$ for any $c\in \C$, we have by definition of $\Omega_k(P)$,
\begin{align}
&\int_{P\setminus F_{j+1}} \absb{T_b^m(f\chi_{3P})}|g|\notag\\
&\hspace{1cm}\le \sum_{k=0}^m{m\choose k}\int_{P\setminus F_{j+1}}|T\big((b-\langle b\rangle_{3P})^{m-k}f\big)||b-\langle b\rangle_{3P}|^{k}|g|\label{ft}\\
&\hspace{1cm} \le C_1\sum_{k=0}^m\big\langle|b-\langle b\rangle_{3P}|^{m-k}|f|\big\rangle_{r,3P}\big\langle |b-\langle b\rangle_{3P}|^{k}|g|\big\rangle_{1,3P}|P|,\notag
\end{align}
where $C_1:=2^m3^n\f_{T,r}(\tfrac{1}{(m+1)6^{n+2}})$.

Now consider the second term in \eqref{eq:toproveend}. Fix $P' \in \mc{F}_{j+1}$ such that $P'\subseteq P$ and denote
$$\psi_k(x):=T\big((b-\langle b\rangle_{3P})^{m-k}f\chi_{3P\setminus 3P'}\big)(x), \qquad x \in \R^n.$$
Then, for $y\in P'$ to be specified later  we have
\begin{equation}
\begin{aligned}
\int_{P'}&\absb{T_b^m(f\chi_{3P \setminus 3P'})} \abs{g}\le 2^m\sum_{k=0}^m\int_{P'}|\psi_k||b-\langle b\rangle_{3P}|^{k}|g|\\
&\le 2^m\sum_{k=0}^m\int_{P'}|\psi_k(x)-\psi_k(y)||b(x)-\langle b\rangle_{3P}|^{k}|g(x)|\dd x\label{yterm}\\
&\hspace{1cm}+2^m\sum_{k=0}^m|\psi_k(y)|\int_{P'}|b(x)-\langle b\rangle_{3P}|^{k}|g(x)|\dd x.
\end{aligned}
\end{equation}
Denote
$$\xi_k(x):=\hab{b(x)-\langle b\rangle_{3P}}^{m-k}f(x)\chi_{3P'}(x),\qquad x \in \R^n,$$
and consider the sets
$$\widetilde\Omega_k(P'):=\{x\in P':|T(\xi_k)(x)|>\f_{T,r}(\tfrac{1}{4(m+1)3^n})\langle|\xi_k|\rangle_{r,3P'}\}.$$
Set $\widetilde\Omega(P'):=\cup_{k=0}^m\widetilde\Omega_k(P')$, for which we have $|\widetilde\Omega(P')|\le \frac{1}{4}|P'|$.
Now, define the good part of the cube $P'$ as
$$G_{P'}:=P'\setminus \hab{\O(P)\cup\widetilde \Omega(P')}.$$
Then, by (\ref{cz}), we have
$$
|G_{P'}|\ge \big(\tfrac{1}{2}-\tfrac{1}{4}\big)|P'|=\tfrac{1}{4}|P'|,
$$
and for all $y\in G_{P'}$ we have
\begin{align*}
|\psi_k(y)|&\le \absb{T\big((b-\langle b\rangle_{3P})^{m-k}f\chi_{3P}\big)(y)}+\absb{T\big((b-\langle b\rangle_{3P})^{m-k}f\chi_{3P'}\big)(y)}\\
&\le \f_{T,r}(\tfrac{1}{(m+1)6^{n+2}})\langle|b-\langle b\rangle_{3P}|^{m-k}|f|\rangle_{r,3P}\\
&\hspace{1cm}+\f_{T,r}(\tfrac{1}{4(m+1)3^n})\langle|b-\langle b\rangle_{3P}|^{m-k}|f|\rangle_{r,3P'}.
\end{align*}
Further, by the definition of $M_k(P)$,
$$\langle|b-\langle b\rangle_{3P}|^{m-k}|f|\rangle_{r,3P'}\le c_{n,m}\langle|b-\langle b\rangle_{3P}|^{m-k}|f|\rangle_{r,3P}.$$
Hence, for all $y\in G_{P'}$, we have
$$|\psi_k(y)|\le 2c_{n,m}\,\f_{T,r}(\tfrac{1}{(m+1)6^{n+2}})\langle|b-\langle b\rangle_{3P}|^{m-k}|f|\rangle_{r,3P}.$$
From this, integrating (\ref{yterm}) over $y\in G_{P'}$, using H\"older's inequality and the definition of the set ${\mathcal M}(P)$,
we obtain
\begin{align*}
&\int_{P'}\absb{T_b^m(f\chi_{3P \setminus 3P'})}|g|\\
&\le 4\cdot 2^m \sum_{k=0}^m\osc_s\big(T(\eta_{m-k}\chi_{{\mathbb R}^n\setminus 3P'});P'\big)\big\langle |b-\langle b\rangle_{3P}|^{k}|g|\big\rangle_{s',P'}|P'|\\
&\hspace{0.1cm}+2c_{n,m}\,\f_{T,r}(\tfrac{1}{(m+1)6^{n+2}})\sum_{k=0}^m\langle|b-\langle b\rangle_{3P}|^{m-k}|f|\rangle_{r,3P}\big\langle |b-\langle b\rangle_{3P}|^{k}|g|\big\rangle_{1,P'}|P'|\\
&\le C_2\sum_{k=0}^m\big\langle |b-\langle b\rangle_{3P}|^{m-k}|f|\big\rangle_{r,3P}\big\langle |b-\langle b\rangle_{3P}|^{k}|g|\big\rangle_{s',P'}|P'|,
\end{align*}
where
$$C_2:=\tilde{c}_{n,m}\big(\f_{T,r}(\tfrac{1}{(m+1)6^{n+2}})+\f_{\mc{M}^{\#}_{T,s},r}(\tfrac{1}{(m+1)6^{n+2}})\big).$$

By H\"older's inequality, for any $q \in [1,\infty)$,
$$\sum_{P' \in \mc{F}_{j+1} : P' \subseteq P}\langle|h|\rangle_{q,P'}|P'|\le \langle|h|\rangle_{q,P}|P|.$$
Therefore,
\begin{align*}
&\sum_{P' \in \mc{F}_{j+1} : P' \subseteq P}\int_{P'}\absb{T_b^m(f\chi_{3P \setminus 3P'})}|g|\\
&\hspace{1cm}\le C_2\sum_{k=0}^m\big\langle |b-\langle b\rangle_{3P}|^{m-k}|f|\big\rangle_{r,3P}\big\langle |b-\langle b\rangle_{3P}|^{k}|g|\big\rangle_{s',3P}|P|,
\end{align*}
which, along with (\ref{ft}), proves (\ref{eq:toproveend}). This completes the proof.
\end{proof}

\begin{remark}\label{dyadversion}
Under the assumptions of Theorem \ref{sdp} and by the ``three lattice theorem" (see, e.g., \cite{LN19}),
there exist $3^n$ dyadic lattices ${\mathscr D}_j$ so that for any $f,g\in L^{\infty}_c({\mathbb R}^n)$ and $b\in L^1_{\text{loc}}({\mathbb R}^n)$, there
exist sparse families ${\mathcal S}_j\subset {\mathscr D}_j$ such that
\begin{align*}
\int_{{\mathbb R}^n}|T_b^mf||g|&\le C\sum_{j=1}^{3^n}\Bigl(\sum_{Q\in {{\mathcal S}_j}}\big\langle |b-\langle b\rangle_Q|^{m}|f|\big\rangle_{r,Q}\big\langle|g|\big\rangle_{s',Q}|Q|\\
&\hspace{1cm}+\sum_{Q\in {{\mathcal S}_j}}\big\langle |f|\big\rangle_{r,Q}\big\langle |b-\langle b\rangle_Q|^{m}|g|\big\rangle_{s',Q}|Q|\Bigr),
\end{align*}
where $C$ is given by (\ref{constant}).
\end{remark}

\section{Weighted estimates for  fractional sparse forms}\label{sec:fracsparse}
In the next section, we will prove quantitative Bloom weighted estimates for the sparse forms in the conclusion of Theorem \ref{sdp}. As a preparation, we establish a weighted estimate for fractional sparse forms in this section.

\begin{theorem}\label{sparsefrac}
Let $1\le r<p\leq q<s\le \infty$, set $\a = \frac{r}p-\frac{r}q$ and let $w \in A_{q/r} \cap {\rm RH}_{(s/q)'}$.
For any sparse family of cubes ${\mathcal S}\subset {\mathscr D}$, $f \in L^p(w^{p/q})$ and $g \in L^{q'}(w^{1-q'})$ we have
$$
\sum_{Q\in {\mathcal S}}\langle|f|\rangle_{\frac{r}{1+\a},Q}\langle|g|\rangle_{s',Q}|Q|^{1+\frac{\a}{r}}\lesssim [w]_{A_{q/r}}^{\beta}[w]_{{\rm{RH}}_{(s/q)'}}^{\beta}\|f\|_{L^p(w^{p/q})}\|g\|_{L^{q'}(w^{1-q'})}
$$
with
$$\beta={\max\Bigl( \tfrac{s(1-\tfrac{\alpha}{r})-1}{s-q},\tfrac{1}{q-r}\Bigr)}.$$
\end{theorem}

For $\a=0$, this theorem was proved by Bernicot--Frey--Petermichl {\cite[Proposition~6.4]{BFP16}}. The general case $\a\ge 0$ is a combination of generalizations by Li \cite{Li17} and Fackler--Hyt\"onen \cite{FH18}.

\begin{remark}\label{rem:symmweight}
  Following the notation of Nieraeth \cite{Ni19}, for $0<r<p<s\leq \infty$ and a weight $w$, define
  $$
  [w]_{p,(r,s)} := \sup_{Q \in \mc{Q}}\, \ip{w^{-1}}_{\frac{pr}{p-r},Q} \ip{w}_{\frac{sp}{s-p},Q}.
  $$
 Upon inspection of the proof, it is clear that the conclusion of Theorem~\ref{sparsefrac} can be more symmetrically phrased as
 $$
\sum_{Q\in {\mathcal S}}\langle|f|\rangle_{\frac{r}{1+\a},Q}\langle|g|\rangle_{s',Q}|Q|^{1+\frac{\a}{r}}\lesssim [w]_{{q,(r,s)}}^{\beta q}\|f\|_{L^p(w^{p})}\|g\|_{L^{q'}(w^{-q'})}
$$
for all weights $w$ such that $[w]_{{q,(r,s)}}<\infty$.
\end{remark}

As a direct corollary of Theorem \ref{sparsefrac}, in the case $r=1$ and $s= \infty$, we recover \cite[Lemma 3.2]{HLS23}, which is a special case of \cite[Theorem 1.1]{FH18}:

\begin{cor}\label{corollary:weightedsparse}
Let $1<p \leq q <\infty, $ set $\a = \frac1p-\frac1q$ and let $w \in A_{q}$.
For any sparse family of cubes ${\mathcal S}\subset {\mathscr D}$ and $f \in L^p(w^{p/q})$ we have
\begin{align*}
  \nrms{\sum_{Q\in {\mathcal S}}\langle|f|\rangle_{\frac{1}{1+\a},Q}|Q|^{{\a}} \chi_Q}_{L^q(w)} \lesssim [w]_{A_q}^{\max\ha{1-\alpha,\frac{1}{q-1}}} \nrm{f}_{L^p(w^{p/q} )} .
\end{align*}
In particular, we recover the  well-known bound
\begin{equation}\label{spbound}
 \nrms{\sum_{Q\in {\mathcal S}}\langle|f|\rangle_{1,Q} \chi_Q}_{L^q(w)}\lesssim [w]_{A_q}^{\max(1,\frac{1}{q-1})}\nrm{f}_{L^q(w )}.
\end{equation}
\end{cor}

The proof of Theorem \ref{sparsefrac} is based on three main ingredients, the first of which is
a very slight generalization of a result of Li {\cite{Li17}}.

\begin{theorem}[{\cite[Theorem 1.2]{Li17}}]\label{twoweight}
Let $1<p\le q<s\leq\infty$ and $r \in (0,p)$. Let $w$ and $\si$ be weights and $\lambda_Q\geq 0$ for any $Q \in \ms{D}$. Let $\mc{S}\subset \ms{D}$ be a sparse family of cubes and suppose that ${\mathcal N}$ is the best constant such that
$$\sum_{Q\in {\mc{S}}}\langle |f|\rangle_{r,Q}\langle |g|\rangle_{s',Q}\la_Q\le {\mathcal N}\|f\|_{L^p(w)}\|g\|_{L^{q'}(\si)}.$$
Denote $u:=w^{\frac{r}{r-p}}$, $v:=\si^{\frac{s'}{s'-q'}}$, set
$$\tau_Q:=\langle u\rangle_Q^{\frac{1}{r}-1}\langle v\rangle_Q^{-\frac{1}{s}}\frac{\la_Q}{|Q|},\qquad Q \in \ms{D}$$
and for $R \in \ms{D}$ define
$$T_R f:=\sum_{Q\in {\mc{S}}:Q\subseteq R}\tau_Q\langle f\rangle_Q\chi_Q.$$
Then
\begin{equation}\label{testing}
{\mathcal N}\eqsim \sup_{R\in {\mc{S}}}\frac{\|T_R(u)\|_{L^q(v)}}{u(R)^{1/p}}+\sup_{R\in {\mc{S}}}\frac{\|T_R(v)\|_{L^{p'}(u)}}{v(R)^{1/q'}}.
\end{equation}
\end{theorem}

\begin{proof}
  In the case $r \geq 1$, the theorem is exactly  \cite[Theorem 1.2]{Li17}. The case $r<1$ is proven analogously. Indeed, only the proof of the equivalence  \cite[(2.1)]{Li17} $\Leftrightarrow$ \cite[(2.2)]{Li17} needs to be adapted. To handle the average of $f$ in the implication $\Leftarrow$, one replaces the maximal operator argument by H\"older's inequality. Conversely, for the implication $\Rightarrow$, one replaces H\"older's inequality by a maximal operator argument, using the boundedness of $ M_{1,u}^{\mc{S}}$ on $L^p(u)$.
\end{proof}

In order to estimate the two terms in \eqref{testing}, we will use the following norm equivalence from Cascante--Ortega--Verbitsky \cite{COV04}.

\begin{lemma}[\cite{COV04}]\label{cov}
Let $p \in [1,\infty)$, let $w$ be a weight and let $\lambda_Q \geq 0$ for all $Q\in \ms{D}$. Then  $$
\Big\|\sum_{Q\in {\mathscr D}}\la_Q\chi_Q\Big\|_{L^p(w)}\eqsim \has{\sum_{Q\in {\mathscr D}}\la_Q\Big(\frac{1}{w(Q)}\sum_{Q'\in {\mathscr D}, Q'\subseteq Q}\la_{Q'}w(Q')\Big)^{p-1}w(Q)}^{1/p}.
$$
\end{lemma}

The final ingredient in the proof of Theorem \ref{sparsefrac} is the following result from Fackler--Hyt\"onen \cite{FH18}.

\begin{lemma}[{\cite[Lemma 4.2]{FH18}}]\label{fh}
Let $w$ and $\si$ be weights and $\a,\b,\ga\ge 0$ with $\a>0$ and $\a+\b+\g\ge 1$. Then we have for any sparse family ${\mathcal S}\subset {\mathscr D}$ and  $R \in \ms{D}$
$$\sum_{Q\in {\mathcal S}:Q\subseteq R}|Q|^{\a}\si(Q)^{\b}w(Q)^{\ga}\lesssim |R|^{\a}\si(R)^{\b}w(R)^{\ga}.$$
\end{lemma}

Combining these three ingredients with the proof strategy from \cite{FH18}, we can now prove Theorem \ref{sparsefrac}.

\begin{proof}[Proof of Theorem \ref{sparsefrac}] A direct computation shows that, in the notation of Theorem~\ref{twoweight}, we have $u=w^{-\frac{r}{q-r}}$, $v=w^{\frac{s}{s-q}}$ and
$$\tau_Q=\langle w^{-\frac{r}{q-r}}\rangle_Q^{\frac{1+\a}{r}-1}\langle w^{\frac{s}{s-q}}\rangle_Q^{-\frac{1}{s}}|Q|^{\frac{\a}{r}}.$$

Let us first consider first the testing condition in (\ref{testing}). We will show that
\begin{equation}\label{firsttest}
\|T_R(u)\|_{L^q(v)}\lesssim [w^{\frac{s}{s-q}}]_{A_{\frac{s}{s-q}\frac{q-r}{r}+1}}^{\frac{1}{q}-\frac{1}{s}}[w^{-\frac{r}{q-r}}]_{A_{\infty}}^{\frac{1}{q}}u(R)^{1/p}.
\end{equation}
By Lemma \ref{cov}, we have
$$
\|T_R(u)\|_{L^q(v)}^q\eqsim \sum_{Q\in {\mathcal S}:Q\subseteq R}
\langle w^{\frac{s}{s-q}}\rangle_Q^{2-q-\frac{1}{s}}\langle w^{-\frac{r}{q-r}}\rangle_Q^{\frac{1+\a}{r}}|Q|^{2-q+\frac{\a}{r}}\Psi(Q)^{q-1},$$
where
\begin{equation}\label{PsiQ}
  \Psi(Q):=
\sum_{Q'\in {\mathcal S}:Q'\subseteq Q}
\langle w^{\frac{s}{s-q}}\rangle_{Q'}^{1-\frac{1}{s}}\langle w^{-\frac{r}{q-r}}\rangle_{Q'}^{\frac{1+\a}{r}}|Q'|^{1+\frac{\a}{r}}.
\end{equation}
For $\d>0$ we have
\begin{align*}
&\Psi(Q)\le [w^{\frac{s}{s-q}}]_{A_{\frac{s}{s-q}\frac{q-r}{r}+1}}^{\d}\\
&\hspace{1.5cm}\cdot\sum_{Q'\in {\mathcal S}:Q'\subseteq Q}\langle w^{\frac{s}{s-q}}\rangle_{Q'}^{1-\frac{1}{s}-\d}\langle w^{-\frac{r}{q-r}}\rangle_{Q'}^{\frac{1+\a}{r}-\d\frac{s}{s-q}\frac{q-r}{r}}|Q'|^{1+\frac{\a}{r}}.
\end{align*}

Our goal now is to use Lemma \ref{fh}. Its assumptions imply the following restrictions on $\d$:
\begin{eqnarray*}
1-\tfrac{1}{s}-\d\ge 0&\Leftrightarrow& \d\le 1-\tfrac{1}{s}\\
\tfrac{1+\a}{r}-\d\tfrac{s}{s-q}\tfrac{q-r}{r}\ge 0&\Leftrightarrow&\d\le \tfrac{1+\a}{q-r}\cdot\tfrac{s-q}{s}\\
\tfrac{1}{s}-\tfrac{1}{r}+\d\bigl(1+\tfrac{s}{s-q}\tfrac{q-r}{r}\bigr)>0&\Leftrightarrow&\d> \tfrac{1}{q}-\tfrac{1}{s}.
\end{eqnarray*}
Moreover, the assumption $\alpha+\beta+\gamma \geq 1$  of Lemma \ref{fh} holds trivially because $\alpha+\beta+\gamma=1+\frac{\a}{r}$.
We conclude that $\d \in \big(\frac{1}{q}-\frac{1}{s},1-\frac{1}{s}\big]$ and the set of $\d$ satisfying all restrictions will be non-empty if
$$\tfrac{1}{q}-\tfrac{1}{s}<\tfrac{1+\a}{q-r}\cdot\tfrac{s-q}{s}.$$
It is easily seen that this estimate is true for all $r<q<s$ and $\a\ge 0$.

Taking $\d>0$ satisfying the above restrictions and applying Lemma~\ref{fh}, we obtain
$$
\Psi(Q)\lesssim [w^{\frac{s}{s-q}}]_{A_{\frac{s}{s-q}\frac{q-r}{r}+1}}^{\d}
\langle w^{\frac{s}{s-q}}\rangle_{Q}^{1-\frac{1}{s}-\d}\langle w^{-\frac{r}{q-r}}\rangle_{Q}^{\frac{1+\a}{r}-\d\frac{s}{s-q}\frac{q-r}{r}}|Q|^{1+\frac{\a}{r}}.
$$
Therefore,
\begin{align*}
\|T_R(u)\|_{L^q(v)}^q&\lesssim [w^{\frac{s}{s-q}}]_{A_{\frac{s}{s-q}\frac{q-r}{r}+1}}^{\d(q-1)}
\sum_{Q\in {\mathcal S}:Q\subseteq R} \langle w^{\frac{s}{s-q}}\rangle_{Q}^{\frac{s-q}{s}-\d(q-1)} \\ &\hspace{2cm}\cdot\langle w^{-\frac{r}{q-r}}\rangle_Q^{q\frac{1+\a}{r}+(q-1)\d\frac{s}{s-q}\frac{q-r}{r}}|Q|^{1+\frac{q\a}{r}}\\
&\leq [w^{\frac{s}{s-q}}]_{A_{\frac{s}{s-q}\frac{q-r}{r}+1}}^{\frac{s-q}{s}}\sum_{Q\in {\mathcal S}:Q\subseteq R}
\langle w^{-\frac{r}{q-r}}\rangle_Q^{1+\frac{q\a}{r}}|Q|^{1+\frac{q\a}{r}}.
\end{align*}
Further,
\begin{align*}
\sum_{Q\in {\mathcal S}:Q\subseteq R} \langle w^{-\frac{r}{q-r}}\rangle_Q^{1+\frac{q\a}{r}}|Q|^{1+\frac{q\a}{r}}
&\leq \Big(\int_R w^{-\frac{r}{q-r}}\Big)^{\frac{q\a}{r}}\sum_{Q\in {\mathcal S}:Q\subseteq R}\int_Qw^{-\frac{r}{q-r}}\\
&\lesssim [w^{-\frac{r}{q-r}}]_{A_{\infty}}\Big(\int_Rw^{-\frac{r}{q-r}}\Big)^{1+\frac{q\a}{r}},
\end{align*}
which, along with the previous estimate, proves (\ref{firsttest}).

\medskip

Now consider the second testing condition in (\ref{testing}). Let us show that
\begin{equation}\label{secondtest}
\|T_R(v)\|_{L^{p'}(v)}\lesssim  [w^{\frac{s}{s-q}}]_{A_{\frac{s}{s-q}\frac{q-r}{r}+1}}^{\frac{1}{q} -\frac{1}{s}}[w^{\frac{s}{s-q}}]_{A_{\infty}}^{\frac{1}{p'}}v(R)^{1/q'}.
\end{equation}
Again by Lemma \ref{cov}, we have
$$\|T_R(v)\|_{L^{p'}(u)}^{p'}\eqsim \sum_{Q\in {\mathcal S}:Q\subseteq R}
\langle w^{\frac{s}{s-q}}\rangle_Q^{1-\frac{1}{s}}\langle w^{-\frac{r}{q-r}}\rangle_Q^{\frac{1+\a}{r}+1-p'}|Q|^{2-p'+\frac{\a}{r}}\Psi(Q)^{p'-1},$$
where $\Psi(Q)$ is as in \eqref{PsiQ}. By the above estimate for $\Psi(Q)$,
\begin{align*}
\|T_R(v)\|_{L^{p'}(u)}^{p'}&\lesssim [w^{\frac{s}{s-q}}]_{A_{\frac{s}{s-q}\frac{q-r}{r}+1}}^{\d(p'-1)}
\sum_{Q\in {\mathcal S}:Q\subseteq R} \langle w^{\frac{s}{s-q}}\rangle_{Q}^{p'(1-\frac{1}{s})-\d(p'-1)}\\
&\hspace{2cm}\cdot\langle w^{-\frac{r}{q-r}}\rangle_Q^{p'\frac{1+\a}{r}-(p'-1)(1+\d\frac{s}{s-q}\frac{q-r}{r})}|Q|^{1+\frac{p'\a}{r}}\\
&\lesssim [w^{\frac{s}{s-q}}]_{A_{\frac{s}{s-q}\frac{q-r}{r}+1}}^{p'(\frac{1}{q}-\frac{1}{s})}\sum_{Q\in {\mathcal S}:Q\subseteq R}
\langle w^{\frac{s}{s-q}}\rangle_{Q}^{1+\frac{p'\a}{r}}|Q|^{1+\frac{p'\a}{r}}\\
&\lesssim [w^{\frac{s}{s-q}}]_{A_{\frac{s}{s-q}\frac{q-r}{r}+1}}^{p'(\frac{1}{q}-\frac{1}{s})}[w^{\frac{s}{s-q}}]_{A_{\infty}}\Big(\int_Rw^{\frac{s}{s-q}}\Big)^{1+\frac{p'\a}{r}},
\end{align*}
from which (\ref{secondtest}) follows.

Combining the two estimates for the testing conditions, inequalities (\ref{firsttest}) and (\ref{secondtest}), we obtain
\begin{align*}
\mc{N}&\lesssim
[w^{\frac{s}{s-q}}]_{A_{\frac{s}{s-q}\frac{q-r}{r}+1}}^{\frac{1}{q}-\frac{1}{s}}[w^{-\frac{r}{q-r}}]_{A_{\infty}}^{\frac{1}{q}}
+[w^{\frac{s}{s-q}}]_{A_{\frac{s}{s-q}\frac{q-r}{r}+1}}^{\frac{1}{q}-\frac{1}{s}}[w^{\frac{s}{s-q}}]_{A_{\infty}}^{\frac1{p'}}\\
&\lesssim [w^{\frac{s}{s-q}}]_{A_{\frac{s}{s-q}\frac{q-r}{r}+1}}^{\frac{1}{q}-\frac{1}{s}+\frac{1}{q}\frac{s-q}{s}\frac{r}{q-r}}
+[w^{\frac{s}{s-q}}]_{A_{\frac{s}{s-q}\frac{q-r}{r}+1}}^{\frac{1}{p'}+\frac{1}{q}-\frac{1}{s}},
\end{align*}
where $\mc{N}$ is as in Theorem \ref{twoweight}.
From this, since
$$[w^t]_{A_{t(q-1)+1}}\le [w]_{A_q}^t[w]_{\rm{RH}_t}^t,$$
we obtain the conclusion with
\begin{align*}
\beta &= {\tfrac{s}{s-q}\max(\tfrac{1}{p'}+\tfrac{1}q-\tfrac{1}{s},\tfrac{1}{q}-\tfrac{1}{s}+\tfrac{1}{q}\tfrac{s-q}{s}\tfrac{r}{q-r})}\\
&={\max\bigl( \tfrac{s(1-\tfrac{\alpha}{r})-1}{s-q}, \tfrac{1}{q-r}\bigr)}.\qedhere
\end{align*}
\end{proof}

\section{Bloom weighted bounds for sparse forms associated to commutators}\label{sec:sparseB}
In this section we consider one of the sparse forms in the conclusion of Theorem \ref{sdp}, namely, ${\mathcal B}_{{\mathcal S},b,r,s}^m(f,g)$ as defined in the introduction.

Let us start with some definitions. Given  $b \in L^1_{\loc}(\R^n)$, a weight $\nu$ and $\a\ge 0$, define the weighted, fractional $\BMO$-seminorm as
$$\|b\|_{\BMO_{\nu}^{\a}}:=\sup_{Q \in \mc{Q}}\frac{1}{\nu(Q)^{1+\frac{\a}{n}}}\int_Q|b-\langle b\rangle_Q|.$$
We omit $\alpha$ from our notation if $\alpha=0$. Furthermore, given a cube $Q \in \mc{Q}$, define the oscillation
$$
\O_\nu(b,Q) :=  \frac{1}{\nu(Q)} \int_Q \abs{b-\ip{b}_Q}
$$
and the weighted sharp maximal function
$$
M^{\#}_\nu (b):=\sup_{Q \in \mc{Q}} \,\O_\nu(b,Q) \chi_Q.
$$
Note that
$
\nrm{b}_{\BMO_\nu} = \nrm{M^{\#}_\nu (b)}_{L^\infty(\R^n)}.
$

\begin{theorem}\label{pleq}
  Let $1\leq r<p,  q<s\le\infty$, $m \in \N$ and $b \in L^1_{\loc}(\R^n)$. Assume that $\mu \in A_{p/r}$ and $\lambda \in A_{q/r}\cap {\rm{RH}}_{(s/q)'}$. Set
  $$\alpha:= -\tfrac{1}{t}:= \tfrac{1}{pm}-\tfrac1{qm},$$
   $\alpha_+ := \max\cbrace{\alpha,0}$ and define the Bloom weight
  $$\nu^{1+\a}:= \mu^{\frac{1}{pm}}\lambda^{-\frac{1}{qm}}.$$
For any sparse family ${\mathcal S}\subset \ms{D}$, $f \in L^p(\mu)$ and $g \in L^{q'}(\lambda^{1-q'})$ we have
$$
{\mathcal B}_{{\mathcal S},b,r,s}^m(f,g)\lesssim  C(\mu,\lambda) \nrm{f}_{L^p(\mu)}\|g\|_{L^{q'}(\la^{1-q'})}\begin{cases}
  \nrm{b}_{\BMO_{\nu}^{\alpha n}}^m, \quad &p\leq q,\\
  \nrm{M^{\#}_\nu (b)}_{L^t(\nu)}^m, \quad &q\leq p,
\end{cases}
$$
where
\newsavebox{\mycases}
\reqnomode
\begin{align}
  \sbox{\mycases}{$\displaystyle {\hspace{-20pt}C(\mu,\lambda)\leq}\left\{\begin{array}{@{}c@{}}\vphantom{[\mu]_{A_{p/r}}^{\frac{\beta_{\mu_1}}{1}} }\\ \vphantom{[\mu]_{A_{p/r}}^{\frac{\beta_{\mu_1}}{1}}}\end{array}\right.\kern-\nulldelimiterspace$}
  \raisebox{-.5\ht\mycases}[0pt][0pt]
  {\usebox{\mycases}}&[\mu]_{A_{p/r}}^{\beta_{\mu_1}} [\lambda]_{A_{q/r}}^{\beta_{\lambda_1}+\beta_{\lambda_2}}[\la]_{\rm{RH}_{(s/q)'}}^{\beta_{\lambda_2}} \quad&&\text{for any } p,q,m \label{estweight:1} \\
     &[\mu]_{A_{p/r}}^{\beta_{\mu_2}}[\la]_{A_{q/r}}^{\beta_{\lambda_2}} [\lambda]_{{\rm{RH}}_{(s/q)'}}^{\beta_{\lambda_2}}  &&\text{if }q\leq p,\, m\geq 2\quad \qquad \label{estweight:2}
\end{align}
\leqnomode
with
\begin{align*}
\beta_{\mu_1}&:=\max\bigl(\tfrac{1}{r},\tfrac{1}{p-r}\bigr)+\big(1-\tfrac{1}{rm}\big)(rm-\lfloor rm\rfloor)\max\big(\tfrac{1-\a_+}{r-\a p},\tfrac{1}{(1+\a)p-r}\big)\\
&\hspace{2.85cm}+\tfrac{q}{p}\sum_{j=1}^{\lfloor rm\rfloor-1}\tfrac{j}{rm}\max\big(\tfrac{1-\a_+}{r+jq\a}, \tfrac{1}{(1-j\a)q-r}\big),\\
\beta_{\mu_2} &:= \tfrac{rm}{p-r}\\
\beta_{\la_1}&:= \tfrac{p}{q}\tfrac{1}{rm}(rm-\lfloor rm\rfloor)\max\big(\tfrac{1-\a_+}{r-\a p},\tfrac{1}{(1+\a)p-r}\big)\\
&\hspace{2.85cm}+\sum_{j=1}^{\lfloor rm\rfloor-1}\big(1-\tfrac{j}{rm}\big)\max\big(\tfrac{1-\a_+}{r+jq\a}, \tfrac{1}{(1-j\a)q-r}\big),\\
\beta_{\lambda_2}&:=\max\bigl(\tfrac{s(1-\tfrac{\alpha}{r})-1}{s-q},\tfrac{1}{q-r}\bigr).
\end{align*}
\end{theorem}

\begin{remark}\label{expl}
Observe that the sense of (\ref{estweight:2}) is that in the case $q\le p$ and  $m\ge 2$ it provides an additional bound
for $C(\mu,\lambda)$, which is incomparable with (\ref{estweight:1}), in general. See Section \ref{sec:application} for a further discussion of this phenomenon.
\end{remark}

Before turning to the proof, let us discuss some particular cases of Theorem \ref{pleq}.

\begin{remark}\label{partcases}
Suppose that $s=\infty$ and thus $[\la]_{\rm{RH}_{(s/q)'}}=1$.
In the diagonal case $p=q$, we have $\a=0$. So, in this case,
\begin{align*}
  \beta_{\mu_1}&=\bigl(rm-\lfloor rm\rfloor+\tfrac{\lfloor rm\rfloor}{2rm}(1+\lfloor rm\rfloor)\bigr)\max\big(\tfrac{1}{r},\tfrac{1}{p-r}\big),\\
  \beta_{\la_1}&=\bigl(\lfloor rm\rfloor-\tfrac{\lfloor rm\rfloor}{2rm}(1+\lfloor rm\rfloor)\bigr)\max\big(\tfrac{1}{r},\tfrac{1}{p-r}\big).
\end{align*}
If we  additionally assume that $rm\in {\mathbb N}$, we get
\begin{align*}
  \beta_{\mu_1}&=\tfrac{rm+1}{2}\max\big(\tfrac{1}{r},\tfrac{1}{p-r}\big),\\
  \beta_{\la_1}&=\tfrac{rm-1}{2}\max\big(\tfrac{1}{r},\tfrac{1}{p-r}\big).
\end{align*}
In particular, for $r=1$, we have
$$\beta_{\mu_1}=\beta_{\la_1}+\beta_{\la_2} =\tfrac{m+1}{2}\max\big(1,\tfrac{1}{p-1}\big),$$
Therefore, if $r=1$, $s=\infty$ and $p=q$, we have
\begin{equation*}
  C\ha{\mu,\lambda}\leq \begin{cases}
  \hab{[\mu]_{A_{p}} [\lambda]_{A_{p}}}^{\frac{m+1}{2}\max(1,\frac{1}{p-1})},\quad &m\geq 1,\\
  [\mu]_{A_{p}}^{\frac{m}{p-1}}[\la]_{A_{p}}^{\max\ha{1,\frac{1}{p-1}}},\quad  &m\geq 2,
\end{cases}
\end{equation*}
the first of which was obtained in \cite{LOR19}. Note that the second estimate is, in general, incomparable to the first.
\end{remark}

\begin{remark}
  In the spirit of Remark \ref{rem:symmweight}, we note that the conclusion of Theorem \ref{pleq} can be replaced by
  \begin{align*}
    {\mathcal B}_{{\mathcal S},b,r,s}^m(f,g)&\lesssim \nrm{f}_{L^p(\mu^p)}\|g\|_{L^{q'}(\la^{-q'})} \\&\hspace{-1cm}\cdot \begin{cases}
      \nrm{b}_{\BMO_{\nu}^{\alpha n}}^m [\mu]_{p, (r,\infty)}^{p\beta_{\mu_1} } [\lambda]_{q,(r,\infty)}^{q\beta_{\lambda_1}}[\la]_{q,(r,s)}^{q\beta_{\lambda_2}}, &p \leq q,\, m\geq 1,\\
     \nrm{M^{\#}_\nu (b)}_{L^t(\nu)}^m [\mu]_{p, (r,\infty)}^{p\beta_{\mu_1} } [\lambda]_{q,(r,\infty)}^{q\beta_{\lambda_1}}[\la]_{q,(r,s)}^{q\beta_{\lambda_2}},\qquad &q \leq p,\, m\geq 1,\\
     \nrm{M^{\#}_\nu (b)}_{L^t(\nu)}^m [\mu]_{p, (r,\infty)}^{p\beta_{\mu_2} } [\la]_{q,(r,s)}^{q\beta_{\lambda_2}}, &q \leq p,\, m\geq 2,
    \end{cases}
  \end{align*}
  for all weights $\mu$ such that $ [\mu]_{p, (r,\infty)}<\infty$ and weights $\lambda$ such that $[\la]_{q,(r,s)}<\infty$.
\end{remark}

Several statements below will be needed to prove Theorem \ref{pleq}, starting with the following lemma from Rivera-R\'ios and the first and third authors \cite{LOR17}.

\begin{lemma}[{\cite[Lemma 5.1]{LOR17}}]\label{oscsparse}
Let $b\in L^1_{\text{loc}}({\mathbb R}^n)$ and let ${\mathcal S}\subset {\mathscr D}$ be a sparse family. There exists a sparse family ${\mathcal S}'\subset {\mathscr D}$ such that ${\mathcal S}\subseteq {\mathcal S}'$
and for every cube $Q\in {\mathcal S}'$,
$$|b-\langle b\rangle_Q|\chi_Q\lesssim \sum_{P\in {\mathcal S}':P\subseteq Q}\big\langle|b-\langle b\rangle_P|\big\rangle_{P}\chi_P.$$
\end{lemma}

We also need the following additional result from Cascante--Ortega--Verbitsky \cite{COV04}.

\begin{lemma}[{\cite[(2.4)]{COV04}}]\label{coveq}
Let $p \in [1,\infty)$ and  $\lambda_Q \geq 0$ for all $Q \in \ms{D}$. Then we have
  $$
\Big(\sum_{Q\in {\mathscr D}}\la_Q\chi_Q\Big)^p\le p\sum_{Q\in {\mathscr D}}\la_Q\chi_Q\Big(\sum_{Q'\in {\mathscr D}:Q'\subseteq Q}\la_{Q'}\chi_{Q'}\Big)^{p-1}
$$
\end{lemma}

We are now ready to prove Theorem \ref{pleq} in the case $p \leq q$.

\begin{proof}[Proof of \eqref{estweight:1} in Theorem \ref{pleq} in the case  $p\leq q$] By Lemma \ref{oscsparse}, there exists a sparse collection of cubes $\mc{S}\subseteq \mc{S}'\subset {\mathscr D}$
such that for any $Q \in \mc{S}$,
\begin{align*}
\ipb{\abs{b-\ip{b}_Q}^mf}_{r,Q}^r &\lesssim  \frac{1}{\abs{Q}} \int_Q \has{\sum_{P \in \mc{S'}:P\subseteq Q} \frac{1}{\abs{P}}\int_P \abs{b-\ip{b}_P}\chi_P}^{rm} \abs{f}^r  \\
&\lesssim \nrm{b}_{\BMO_\nu^{\alpha n}}^{rm} \frac{1}{\abs{Q}} \int_Q \has{\sum_{P \in \mc{S}':P\subseteq Q}\frac{\nu(P)^{1+\alpha}}{\abs{P}}\chi_P }^{rm}  \abs{f}^r.
\end{align*}

Let $k:=\lfloor rm\rfloor$ and $\gamma:=rm-(k-1) \in [1,2)$. Applying subsequently Lemma \ref{coveq} $(k-1)$ times yields
\begin{align*}
&\int_Q \has{\sum_{P \in \mc{S}':P\subseteq Q}\frac{\nu(P)^{1+\alpha}}{\abs{P}}\chi_P }^{rm}  \abs{f}^r\\
&\lesssim
\sum_{P_{k-1} \subseteq \cdots\subseteq P_1 \subseteq Q} \frac{\nu(P_1)^{1+\alpha}}{\abs{P_1}} \cdots \frac{\nu(P_{k-1})^{1+\alpha}}{\abs{P_{k-1}}}
\int_{P_{k-1}} \has{\sum_{P_k \subseteq P_{k-1}}\frac{\nu(P_k)^{1+\alpha}}{\abs{P_k}}\chi_{P_k} }^{\gamma}\abs{f}^r,
\end{align*}
where we omitted the assumption $P_1,\ldots,P_k \in \mc{S}'$ from our notation for brevity.

For $\delta \geq 0$ we denote
  \begin{equation*}
   \mc{A}_{\mc{S}',\delta}(\varphi) := \sum_{Q\in {\mathcal S}'}\langle|\varphi|\rangle_{\frac{1}{1+\delta},Q}|Q|^{{\delta}} \chi_Q,
  \end{equation*}
in which we omit $\delta$ if $\delta = 0$.
Using  Lemma \ref{cov} with the weight $w = \abs{f}^r$ and Minkowski's inequality, we have
\begin{align*}
&\int_{P_{k-1}}\has{\sum_{P_k \subseteq P_{k-1}}\frac{\nu(P_k)^{1+\alpha}}{\abs{P_k}}\chi_{P_k} }^{\gamma}\abs{f}^r\\
&\eqsim
\sum_{P_k\subseteq P_{k-1}}\frac{\nu({P_k})^{1+\alpha}}{\abs{{P_k}}} \has{\int_{P_k} \abs{f}^r}^{2-\gamma} \Big(\sum_{P\subseteq P_k}\frac{\nu(P)^{1+\alpha}}{\abs{P}}\int_{P}\abs{f}^r\Big)^{\gamma-1}\\
&\leq \sum_{P_k\subseteq P_{k-1}}\frac{\nu({P_k})^{1+\alpha}}{\abs{{P_k}}} \has{\int_{P_k} \abs{f}^r}^{2-\gamma} \Big(\int_{P_k}{\mathcal A}_{{\mathcal S}'}(|f|^r)^{\frac{1}{1+\alpha}}\nu\Big)^{(1+\alpha)(\gamma-1)}\\
&=  \sum_{{P_k}\subseteq P_{k-1}}{\nu({P_k})^{1+\alpha}} \ip{\abs{f}^r}_{1,P_k}^{2-\gamma} \cdot \ipb{\mc{A}_{\mc{S}'}(\abs{f}^r)\nu}_{\frac{1}{1+\alpha},P_k}^{\gamma-1}\abs{P_k}^{\alpha(\gamma-1)}\\
&\leq \has{\int_{P_{k-1}} \has{\mc{A}_{\mc{S}'}(\abs{f}^r)^{\gamma-2} \cdot \mc{A}_{\mc{S}',\alpha}\hab{{{\mathcal A}}_{\mc{S}'}(\abs{f}^r) \nu^{1+\alpha} }^{\gamma-1} \cdot \nu^{1+\alpha}}^{\frac{1}{1+\alpha}} }^{1+\alpha}\\
&=: \has{\int_{P_{k-1}} h^{\frac{1}{1+\alpha}} }^{1+\alpha}.
\end{align*}
Let us further write
$$\mc{A}_{\mc{S}',\alpha,\nu}(\varphi):=\mc{A}_{\mc{S}',\alpha}(\varphi)\nu^{1+\alpha},$$
and let $\mc{A}_{\mc{S}',\alpha,\nu}^j$ be the $j$-th iteration of $\mc{A}_{\mc{S}',\alpha,\nu}$. Using Minkowski's inequality $(k-1)$ times more, we find
\begin{align*}
  \sum_{P_{k-1} \subseteq \cdots\subseteq P_1 \subseteq Q} \frac{\nu(P_1)^{1+\alpha}}{\abs{P_1}} \cdots \frac{\nu(P_{k-1})^{1+\alpha}}{\abs{P_{k-1}}} & \has{\int_{P_{k-1}} h^{\frac{1}{1+\alpha}} }^{1+\alpha} \\ &\leq \has{\int_{Q} \hab{\mc{A}^{k-1}_{\mc{S}',\alpha,\nu}(h)}^{\frac{1}{1+\alpha}}}^{1+\alpha},
\end{align*}
which, along with the previous estimates, implies
$$
\ipb{\abs{b-\ip{b}_Q}^mf}_{r,Q}^r\lesssim \nrm{b}_{\BMO_\nu^{\alpha n}}^{rm} \ipb{\mc{A}^{k-1}_{\mc{S}',\alpha,\nu}(h)}_{\frac{1}{1+\alpha},Q}\abs{Q}^\alpha.
$$
From this, we conclude
\begin{align*}
{\mathcal B}_{\mc{S},b,r,s}^m(f,g)&\lesssim \nrm{b}_{\BMO_\nu^{\alpha n}}^{m}\sum_{Q\in {\mathcal S}}\langle (\mc{A}^{k-1}_{\mc{S}',\alpha,\nu}(h))^{\frac1r}\rangle_{\frac{r}{1+\a},Q}\langle|g|\rangle_{s',Q}|Q|^{1+\frac{\a}{r}}.
\end{align*}

Now, for $j=1,\ldots,k-1$, define
\begin{align*}
  \tfrac{1}{u_j} &:=\tfrac{1}q+j\tfrac{\alpha}{r} = \tfrac{j}{rm}\tfrac{1}{p}+(1-\tfrac{j}{rm})\tfrac{1}{q},\\
  w_j&:= \mu^{\frac{j}{rm}\frac{u_j}{p}}\lambda^{(1-\frac{j}{rm})\frac{u_j}{q}}.
\end{align*}
By Theorem \ref{sparsefrac}, we have
\begin{align*}
{\mathcal B}_{\mc{S},b,r,s}^m(f,g)&\lesssim C(\lambda)  \,\nrm{b}_{\BMO_\nu^{\alpha n}}^m \nrm{(\mc{A}^{k-1}_{\mc{S}',\alpha,\nu}(h))^{\frac1r}}_{L^{u_1}(\lambda^{u_1/q})}\|g\|_{L^{q'}(\la^{1-q'})},
\end{align*}
where
$$C(\lambda):=\big([\la]_{A_{q/r}}[\la]_{{\rm{RH}}_{(s/q)'}}\big)^{\max\bigl( \frac{s(1-\frac{\alpha}{r})-1}{s-q},\frac{1}{q-r}\bigr)}.$$
Next, we apply Corollary \ref{corollary:weightedsparse} $(k-1)$ times to obtain
\begin{align*}
\nrm{ \mc{A}^{k-1}_{\mc{S}',\alpha,\nu}(h)^{\frac{1}{r}}}_{L^{u_1}(\lambda^{u_1/q})}^r
&=\|{\mathcal A}_{{\mathcal S}',\alpha}({\mathcal A}_{{\mathcal S}',\alpha,\nu}^{k-2}(h))\|_{L^{u_1/r}(w_1)}\\
&\lesssim [w_1 ]_{A_{u_1/r}}^{\max\ha{1-\alpha, \frac{r}{u_1-r}}} \nrm{\mc{A}^{k-2}_{\mc{S}',\alpha,\nu}(h)}_{L^{u_2/r}(w_1)}\\&\lesssim \cdots\\
&\lesssim \has{\prod_{j=1}^{k-1} [w_j]_{A_{u_j/r}}^{\max\ha{1-\alpha, \frac{r}{u_j-r}}}} \nrm{h}_{L^{u_k/r}(w_{k-1})}.
\end{align*}

Now note that, for $\frac{1}{v} := \frac{1}{p} -\frac{\alpha}{r}$, we have
  $$
  \tfrac{1}{u_k} = \tfrac{1}{q} + (rm-(\gamma-1))\tfrac\alpha{r} = (2-\gamma)\tfrac{1}{p} +(\gamma-1)\tfrac{1}{v}.
  $$
  So, defining
  $$
  w_v:=\mu^{(1-\frac{1}{rm})\frac{v}{p}} \lambda^{\frac{1}{rm}\frac{v}{q}},
  $$
  we have by H\"older's inequality,
  \begin{align*}
&\nrm{h}_{L^{u_k/r}(w_{k-1})}\\
  &=\nrmb{ {\mathcal A}_{\mc{S}'}(\abs{f}^r)^{2-\gamma}   {{\mathcal A}}_{\mc{S}',\alpha}\hab{{{\mathcal A}}_{\mc{S}'}(\abs{f}^r) \nu^{1+\alpha}}^{\gamma-1} \nu^{1+\alpha}}_{L^{u_k/r}( \mu^{(1-\frac{\gamma}{rm})\frac{u_{k}}{p}}\lambda^{\frac{\gamma}{rm}\frac{u_{k}}{q}})}\\
  &=\nrmb{ {\mathcal A}_{\mc{S}'}(\abs{f}^r)^{2-\gamma}   {{\mathcal A}}_{\mc{S}',\alpha}\hab{{{\mathcal A}}_{\mc{S}'}(\abs{f}^r) \nu^{1+\alpha}}^{\gamma-1} \lambda^{\frac{\gamma-1}{qm}} \mu^{\frac{r}{p}-\frac{\gamma-1}{pm}}}_{L^{u_k/r}}\\
  &\leq  \nrmb{{\mathcal A}_{\mc{S}'}(\abs{f}^r)}_{L^{p/r}(\mu)}^{2-\gamma} \nrmb{{{\mathcal A}}_{\mc{S}',\alpha}\hab{{{\mathcal A}}_{\mc{S}'}(\abs{f}^r) \nu^{1+\alpha}}}_{L^{v/r}(w_v)}^{\gamma-1}.
\end{align*}
For the first term on the right-hand side, by (\ref{spbound}), we have
\begin{align*}
\nrmb{{\mathcal A}_{\mc{S}'}(\abs{f}^r) }_{L^{p/r}(\mu)}^{2-\gamma}&\lesssim [\mu]_{A_{p/r}}^{(2-\gamma)\max\ha{1,\frac{r}{p-r}}} \nrm{f}_{L^p(\mu)}^{(2-\gamma)r}.
\end{align*}
For the second term, by applying Corollary \ref{corollary:weightedsparse} and (\ref{spbound}), we have
\begin{align*}
  \nrmb{{{\mathcal A}}_{\mc{S}',\alpha}&\hab{{{\mathcal A}}_{\mc{S}'}(\abs{f}^r) \nu^{1+\alpha}}}_{L^{v/r}(w_v)}^{\gamma-1} \\
  &\lesssim [w_v ]_{A_{v/r}}^{\ha{\gamma-1} \max\ha{1-\alpha,\frac{r}{v-r}}} \nrmb{{{{\mathcal A}}_{\mc{S}'}(\abs{f}^r)}}_{L^{p/r}(\mu)}^{\gamma-1}\\
  &\lesssim[w_v ]_{A_{v/r}}^{\ha{\gamma-1} \max\ha{1-\alpha,\frac{r}{v-r}}} [\mu]_{A_{p/r}}^{(\gamma-1)\max\ha{1,\frac{r}{p-r}}} \nrm{f}_{L^p(\mu)}^{(\gamma-1)r}.
\end{align*}

Collecting our estimates, we have shown
  \begin{align*}
{\mathcal B}_{\mc{S},b,r,s}^m(f,g)\lesssim C\ha{\mu,\lambda} \nrm{b}_{\BMO_{\nu}^{\alpha n}}^m  \nrm{f}_{L^p(\mu)}\|g\|_{L^{q'}(\la^{1-q'})}
\end{align*}
with
\begin{align*}
  C\ha{\mu,\lambda} &:= C(\lambda) \cdot [\mu]_{A_{p/r}}^{\max\ha{\frac1r,\frac{1}{p-r}}} \cdot  [w_v]_{A_{v/r}}^{\ha{\gamma-1} \max\ha{\frac{1-\alpha}{r},\frac{1}{v-r}}}  \\&\hspace{2cm}\cdot \has{\prod_{j=1}^{k-1} [ w_j]_{A_{u_j/r}}^{\max\ha{\frac{1-\alpha}{r}, \frac{1}{u_j-r}}}}.
\end{align*}
By H\"older's inequality, since $\frac{1}{v}=(1-\frac{1}{rm})\frac{1}{p}+\frac{1}{rm}\frac{1}{q}$, we have
$$
[w_v]_{A_{v/r}} = [\mu^{(1-\frac{1}{rm})\frac{v}{p}} \lambda^{\frac{1}{rm}\frac{v}{q}} ]_{A_{v/r}}\le [\mu]_{A_{p/r}}^{(1-\frac{1}{rm})\frac{v}{p}}[\la]_{A_{q/r}}^{\frac{1}{rm}\frac{v}{q}}.
$$
Similarly, we have
$$
[w_j]_{A_{u_j/r}} = [\mu^{\frac{j}{rm}\frac{u_j}{p}}\lambda^{(1-\frac{j}{rm})\frac{u_j}{q}}]_{A_{u_j/r}}\le [\mu]_{A_{p/r}}^{\frac{j}{rm}\frac{u_j}{p}}[\la]_{A_{q/r}}^{(1-\frac{j}{rm})\frac{u_j}{q}}.
$$
From this and from the above expression for $C\ha{\mu,\la}$, the values of $\beta_{\mu_1}$, and $\beta_{\lambda_1}$ follow by direct computation.
\end{proof}

We now turn to the case $q \leq p$. We start with the estimate in \eqref{estweight:1}, which works for any $m\geq 1$.

\begin{proof}[Proof of \eqref{estweight:1} in Theorem \ref{pleq} in the case $q\leq p$]
By Lemma \ref{oscsparse}, there exists a sparse collection of cubes $\mc{S}\subseteq \mc{S}'\subset {\mathscr D}$
such that for any $Q \in \mc{S}$,
\begin{align*}
\ipb{\abs{b-\ip{b}_Q}^mf}_{r,Q}^r &\lesssim  \frac{1}{\abs{Q}} \int_Q \has{\sum_{P \in \mc{S'}:P\subseteq Q} \frac{1}{\abs{P}}\int_P \abs{b-\ip{b}_P}\chi_P}^{rm} \abs{f}^r  \\
&=\frac{1}{\abs{Q}} \int_Q \has{\sum_{P \in \mc{S'}:P\subseteq Q} \frac{\nu(P)}{|P|}\O_{\nu}(b,P)\chi_P}^{rm}\abs{f}^r.
\end{align*}

Let $k:=\lfloor rm\rfloor$ and $\gamma:=rm-(k-1) \in [1,2)$. Applying subsequently Lemma \ref{coveq} $(k-1)$ times yields
\begin{align*}
&\int_Q \has{\sum_{P \in \mc{S}':P\subseteq Q}\frac{\nu(P)}{|P|}\O_{\nu}(b,P)\chi_P }^{rm}  \abs{f}^r\\
&\lesssim
\sum_{P_{k-1} \subseteq \cdots\subseteq P_1 \subseteq Q} \frac{\nu(P_1)}{|P_1|}\O_{\nu}(b,P_1) \cdots \frac{\nu(P_{k-1})}{|P_{k-1}|}\O_{\nu}(b,P_{k-1})\\
&\hspace{2cm}\cdot \int_{P_{k-1}} \has{\sum_{P_k \subseteq P_{k-1}}\frac{\nu(P_k)}{|P_k|}\O_{\nu}(b,P_k)\chi_{P_k} }^{\gamma}\abs{f}^r,
\end{align*}
where we omitted the assumption $P_1,\ldots,P_k \in \mc{S}'$ from our notation for brevity.

Define
\begin{align*}
   \mc{A}_{\mc{S}'}(\varphi) := \sum_{Q\in {\mathcal S}'}\langle|\varphi|\rangle_{{1},Q} \chi_Q,\\
   \mc{A}_{\mc{S}',\nu,b}(\varphi):=\mc{A}_{\mc{S}'}(\varphi)M^{\#}_\nu(b)\nu
\end{align*}
and let $\mc{A}_{\mc{S}',\nu,b}^j$ be the $j$-th iteration of $\mc{A}_{\mc{S}',\nu,b}$.
By Lemma \ref{cov} and H\"older's inequality, we can estimate
\begin{align*}
&\int_{P_{k-1}} \has{\sum_{P_k\subseteq P_{k-1}}\frac{\nu(P_k)}{\abs{P_k}} \O_{\nu}(b,P_{k})\chi_{P_k} }^\gamma  \abs{f}^r\\
&\eqsim \sum_{P_k\subseteq P_{k-1}}\frac{\nu(P_k)}{\abs{P_k}} \O_{\nu}(b,P_{k}) \has{\int_{P_k} \abs{f}^r}^{2-\gamma}\has{\sum_{P\subseteq P_k}\frac{\nu(P)}{\abs{P}} \O_{\nu}(b,P) \int_{P} \abs{f}^r}^{\gamma-1}\\
&\le\sum_{P_k\subseteq P_{k-1}}{\nu (P_k)}\Omega_\nu(b,P_k)\has{\frac{1}{\abs{P_k}}\int_{P_k}\abs{f}^r}^{2-\gamma}\has{\frac{1}{\abs{P_k}}\int_{P_k}{{\mathcal A}}_{{\mc{S}',\nu,b}}(\abs{f}^r)}^{\gamma-1}\\
&\leq \int_{P_{k-1}}{ \mc{A}_{\mc{S}'}(\abs{f}^r)^{2-\gamma} \cdot \mc{A}_{\mc{S}'}({{\mathcal A}}_{{\mc{S}',\nu,b}}(\abs{f}^r))^{\gamma-1}}\cdot M^{\#}_\nu (b)\nu =: \int_{P_{k-1}} h.
\end{align*}
Next, we can iteratively estimate
\begin{align*}
&\sum_{P_{k-1} \subseteq \cdots\subseteq P_1 \subseteq Q} \frac{\nu(P_1)}{|P_1|}\O_{\nu}(b,P_1) \cdots \frac{\nu(P_{k-1})}{|P_{k-1}|}\O_{\nu}(b,P_{k-1})\int_{P_{k-1}}h\\
&\le\sum_{P_{k-2} \subseteq \cdots\subseteq P_1 \subseteq Q} \frac{\nu(P_1)}{|P_1|}\O_{\nu}(b,P_1) \cdots \frac{\nu(P_{k-2})}{|P_{k-2}|}\O_{\nu}(b,P_{k-2})\int_{P_{k-2}}\mc{A}_{\mc{S}',\nu,b}(h)\\
&\le\cdots\le \int_Q\mc{A}_{\mc{S}',\nu,b}^{k-1}(h).
\end{align*}
Combined with the previous estimates, this implies
$$
{\mathcal B}_{\mc{S},b,r,s}^m(f,g)\lesssim \sum_{Q\in {\mathcal S}}\langle\mc{A}_{\mc{S}',\nu,b}^{k-1}(h)^{\frac1r}\rangle_{r,Q}\langle|g|\rangle_{s',Q}|Q|.
$$
From this, using Theorem \ref{sparsefrac}, we obtain
$$
{\mathcal B}_{b,r,s}^m(f,g)\lesssim \big([\la]_{A_{q/r}}[\la]_{{\rm{RH}}_{(s/q)'}}\big)^{\max(\frac{s-1}{s-q},\frac{1}{q-r})}\|\mc{A}_{\mc{S}',\nu,b}^{k-1}(h)\|_{L^{q/r}(\la)}^{1/r}\|g\|_{L^{q'}(\la^{1-q'})}.
$$

Define for $j=1,\ldots,k-1$
\begin{align*}
  \tfrac{1}{u_j} &:= \tfrac{1}{q}- \tfrac{j}{rt} = \tfrac{j}{rm}\tfrac{1}{p}+(1-\tfrac{j}{rm})\tfrac{1}{q},\\
  w_j&:= \mu^{\frac{j}{rm}\frac{u_j}{p}}\lambda^{(1-\frac{j}{rm})\frac{u_j}{q}}.
\end{align*}
Applying H\"older's inequality along with (\ref{spbound}) $(k-1)$ times, we estimate
\begin{align*}
\nrm{\mc{A}^{k-1}_{\mc{S}',\nu,b}(h)}_{L^{q/r}(\lambda)}&= \nrmb{M^{\#}_{\nu}(b)\nu^{1/t}\mc{A}_{{\mc{S}'}}(\mc{A}^{k-2}_{\mc{S}',\nu,b}h)\nu^{1/t'} \lambda^{r/q}}_{L^{q/r}}\\
&\le \nrm{M^{\#}_{\nu}(b)}_{L^t(\nu)} \nrm{\mc{A}_{\mc{S}'}(\mc{A}^{k-2}_{\mc{S}',\nu,b}h)}_{L^{u_1/r}(w_1)}\\
&\lesssim  [w_1]_{A_{u_1/r}}^{\max\ha{1, \frac{r}{u_1-r}}} \nrm{M^{\#}_{\nu}(b)}_{L^t(\nu)}  \nrm{\mc{A}^{k-2}_{\mc{S}',\nu,b}h}_{L^{u_1/r}(w_1)}\\&\lesssim \cdots \\
&\lesssim  \has{\prod_{j=1}^{k-1} [w_j ]_{A_{u_j/r}}^{\max\ha{1, \frac{r}{u_j-r}}}} \nrm{M^{\#}_{\nu}(b)}_{L^t(\nu)}^{k-1}  \nrm{h}_{L^{u_{k-1}/r}(w_{k-1})}.
\end{align*}

Now define $\frac{1}{v} = \frac{1}{p}+\frac1{rt}$ and
$$
w_v:= \mu^{(1-\frac1{rm})\frac{v}{p}}\lambda^{\frac1{rm}\frac{v}{q}}.
$$
  Noting that $1-\frac{k-1}{rm}=\frac{\gamma}{rm}$  and thus
  \begin{align*}
    \tfrac{1}{u_{k-1}} = \tfrac{1}{q} - \tfrac{k-1}{rt}  &= \tfrac{1}{p}+\tfrac{\gamma}{rt}
    =\tfrac{1}{rt} +(2-\gamma)\tfrac{1}{p} + (\gamma-1)\tfrac{1}{v},
  \end{align*}
  we can estimate by H\"older's inequality,
\begin{align*}
\nrm{h}_{L^{u_{k-1}/r}(w_{k-1})}&= \nrm{h}_{L^{u_{k-1}/r}(\mu^{(1-\frac{\gamma}{rm})\frac{u_{k-1}}{p}}\lambda^{\frac{\gamma}{rm}\frac{u_{k-1}}{q}} )}\\
  &\leq   \nrm{M^{\#}_{\nu}(b)}_{L^t(\nu)} \nrmb{\mc{A}_{\mc{S}'}(\abs{f}^r)}_{L^{p/r}(\mu)}^{2-\gamma} \\&\hspace{2cm}\cdot \nrmb{{{\mathcal A}}_{\mc{S}'}\hab{{{\mathcal A}}_{\mc{S}',\nu,b}(\abs{f}^r)}}_{L^{v/r}( w_v)}^{\gamma-1}.
\end{align*}
For the second term on the right-hand side we have, by (\ref{spbound}),
\begin{align*}
   \nrmb{{\mathcal A}_{\mc{S}'}(\abs{f}^r)}_{L^{p/r}(\mu)}^{2-\gamma}\lesssim [\mu]_{A_{p/r}}^{(2-\gamma)\max\ha{1,\frac{r}{p-r}}} \nrm{f}_{L^p(\mu)}^{(2-\gamma)r},
\end{align*}
and for the third term we have
\begin{align*}
  &\nrmb{{{\mathcal A}}_{\mc{S}'}\hab{{{\mathcal A}}_{\mc{S}',\nu,b}(\abs{f}^r)}}_{L^{v/r}( w_v)}^{\gamma-1} \\
  &\lesssim [ w_v ]_{A_{v/r}}^{(\gamma-1)\max\ha{1,\frac{r}{v-r}}} \nrmb{{{\mathcal A}}_{\mc{S}',\nu,b}(\abs{f}^r)}_{L^{v/r}(w_v )}^{\gamma-1}\\
  &\lesssim  [ w_v]_{A_{v/r}}^{(\gamma-1)\max\ha{1,\frac{r}{v-r}}}\nrm{M^{\#}_{\nu}b}_{L^t(\nu)}^{\gamma-1} \nrmb{{{\mathcal A}}_{\mc{S}'}(\abs{f}^r)}_{L^{p/r}(\mu)}^{\gamma-1}\\
  &\lesssim  [ w_v]_{A_{v/r}}^{(\gamma-1)\max\ha{1,\frac{r}{v-r}}}  [\mu]_{A_{p/r}}^{(\gamma-1)\max\ha{1,\frac{r}{p-r}}} \nrm{M^{\#}_{\nu}(b)}_{L^t(\nu)}^{\gamma-1}\nrm{f}_{L^p(\mu)}^{(\gamma-1)r}.
\end{align*}

Collecting our estimates, we have shown
  \begin{align*}
{\mathcal B}_{b,r,s}^m(f,g) \lesssim C\ha{\mu,\lambda}\, [\la]_{{\rm{RH}}_{(s/q)'}}^{\max(\frac{s-1}{s-q},\frac{1}{q-r})}  \nrm{M^{\#}_\nu (b)}_{L^t(\nu)}^m \nrm{f}_{L^p(\mu)}\|g\|_{L^{q'}(\la^{1-q'})}
\end{align*}
with
\begin{align*}
  C\ha{\mu,\lambda} &=  [\mu]_{A_{p/r}}^{\max\ha{\frac1r,\frac{1}{p-r}}}  \cdot [w_v]_{A_{v/r}}^{\ha{\gamma-1}\max\ha{\frac{1}{r},\frac{1}{v-r}}} \\&\hspace{2cm}\cdot \has{\prod_{j=1}^{k-1} [w_j ]_{A_{u_j/r}}^{\max\ha{\frac1r, \frac{1}{u_j-r}}}}\cdot [\lambda]_{A_{q/r}}^{\max(\frac{s-1}{s-q},\frac{1}{q-r})}.
\end{align*}
By H\"older's inequality,
$$
[w_v]_{A_v/r} = [\mu^{(1-\frac1{rm})\frac{v}{p}}\lambda^{\frac1{rm}\frac{v}{q}}]_{A_{v/r}}\le [\mu]_{A_{p/r}}^{(1-\frac1{rm})\frac{rv}{p}}[\la]_{A_{q/r}}^{\frac1m\frac{v}{q}}
$$
and
$$
[w_j]_{A_{u_j/r}}=[\mu^{\frac{j}{rm}\frac{u_j}{p}}\lambda^{(1-\frac{j}{rm})\frac{u_j}{q}} ]_{A_{u_j/r}}\le [\mu]_{A_{p/r}}^{\frac{j}{rm}\frac{u_j}{p}}[\la]_{A_{q/r}}^{(1-\frac{j}{rm})\frac{u_j}{q}}.
$$
From this and from the above expression for $C\ha{\mu,\la}$, the values of $\beta_{\mu_1}$ and $\beta_{\la_1}$ follow by direct computation.
\end{proof}

For the estimate \eqref{estweight:2} in Theorem \ref{pleq} in the case $q \leq p$, we need a Fefferman--Stein-type lemma (see, e.g., \cite{LPRR19}).
\begin{lemma}\label{lem:FS}
  Let $p \in (1,\infty)$, $0<\delta<1$ and let $w$ be a weight. For any sparse family $\mc{S} \subset \ms{D}$ and $f \in L^1_{\loc}(\R^n)$ we have
  $$
\nrms{\sum_{Q\in{\mathcal S}}\langle|f|\rangle_{1,Q}\chi_Q}_{L^p(w)}\lesssim p'\tfrac{1}{\d^{1/p'}}\|f\|_{L^p(M_{1+\d}w)}.
$$
\end{lemma}

Using this lemma, we can now prove the estimate \eqref{estweight:2} in Theorem~\ref{pleq} in the case $q \leq p$. This approach was suggested qualitatively in the case $p=q$ by Li \cite{Li22}.

\begin{proof}[Proof of \eqref{estweight:2} in Theorem \ref{pleq}] By Lemma \ref{oscsparse}, there exists a sparse collection of cubes $\mc{S}\subseteq \mc{S}'\subset {\mathscr D}$
such that for any $Q \in \mc{S}$,
$$\big\langle|b-\langle b\rangle_Q|^m|f|\big\rangle_{r,Q}\lesssim |Q|^{-1/r}\nrms{\sum_{Q\in{\mathcal S}'}\langle M^{\#}_\nu b \cdot \nu \rangle_{1,Q}\chi_Q}_{L^{rm}(|f|^r)}^m.$$
Since $m\geq 2$, we have by Lemma \ref{lem:FS},
$$\big\langle|b-\langle b\rangle_Q|^mf\big\rangle_{r,Q}\lesssim |Q|^{-1/r}\frac{1}{\d^{m/(rm)'}}\|M^{\#}_\nu b\cdot \nu\chi_Q\|_{L^{rm}(M_{1+\d}(|f|^r))}^m.$$
Therefore, by Theorem \ref{sparsefrac}, Buckley's estimate \cite{Bu93} and H\"older's inequality, we have for $0<\delta<1$,
\begin{align*}
{\mathcal B}_{{\mathcal S},b,r,s}^m(f,g)&\lesssim \tfrac{1}{\d^{m/(rm)'}}
\sum_{Q\in{\mathcal S}}\big\langle(M^{\#}_\nu b\cdot \nu)^m M_{(1+\d)r}f \big\rangle_{r,Q}\ip{\abs{g}}_{s',Q} \abs{Q}\\
&\lesssim \tfrac{1}{\d^{m/(rm)'}} \hab{[\la]_{A_{q/r}}[\lambda]_{{\rm{RH}}_{(s/q)'}}}^{\max(\frac{s-1}{s-q},\frac{1}{q-r})}\\&\hspace{2cm}\cdot \|(M^{\#}_\nu b\cdot \nu)^m M_{(1+\d)r}f\|_{L^q(\lambda)} \nrm{g}_{L^{q'}(\lambda^{1-q'})}\\
&\lesssim \tfrac{1}{\d^{m/(rm)'}} \hab{[\la]_{A_{q/r}}[\lambda]_{{\rm{RH}}_{(s/q)'}}}^{\max(\frac{s-1}{s-q},\frac{1}{q-r})}\\&\hspace{2cm}\cdot \| M_{(1+\d)r}f\|_{L^p(\mu)} \nrm{g}_{L^{q'}(\lambda^{1-q'})}\nrm{M^{\#}_\nu (b)}_{L^t(\nu)}^m\\
&\lesssim \tfrac{1}{\d^{m/(rm)'}}[\mu]_{A_{\frac{p}{(1+\d)r}}}^{\frac{1}{p-(1+\d)r}}\hab{[\la]_{A_{p/r}} [\lambda]_{{\rm{RH}}_{(s/q)'}}}^{\max(\frac{s-1}{s-q},\frac{1}{q-r})} \\&\hspace{2cm}\cdot \|f\|_{L^p(\mu)}\nrm{g}_{L^{q'}(\lambda^{1-q'})}\nrm{M^{\#}_\nu (b)}_{L^t(\nu)}^m.
\end{align*}
Let $c_n>0$ be the constant from  Proposition \ref{rhweights} and set $$\tfrac{1}{\d} :={(p/r)'\cdot c_n} \, [\mu]_{A_{p/r}}^{\frac{r}{p-r}}.$$ \eqref{estweight:2} now follows from Proposition \ref{rhweights}.
\end{proof}

\begin{remark}
  The proof of \eqref{estweight:2} in Theorem \ref{pleq} also works in the case $m=1$ and $r>1$ with a constant depending on $r$, as can be seen from the proof. However, in applications to concrete operators, this will yield worse dependence on the weight characteristics of $\mu$ and $\lambda$ than \eqref{estweight:1}.

  The method of proof of \eqref{estweight:2} is likely also applicable to the case $p<q$, again yielding an incomparable bound to \eqref{estweight:1}. This would require one to develop a fractional version of Lemma \ref{lem:FS}. Since our interest in quantitative estimates is mainly in the case $p=q$, we leave this extension to the interested reader.
\end{remark}

\section{Weighted bounds for commutators}\label{sec:application}
In this final section we will apply the results from the previous sections to concrete operators.
Let us first formulate a general result, in a qualitative form, which is an immediate corollary of
Theorems \ref{sdp} and~\ref{pleq}  and duality.

\begin{theorem}\label{wgeneral}
Let $1\le r<p,q< s\le\infty$ and $m\in {\mathbb N}$. Let $T$ be a sublinear operator and $b \in L^1_{\loc}(\R^n)$. Assume the following conditions:
\begin{itemize}
\item Suppose $T$ and ${\mathcal M}^{\#}_{T,s}$ are  locally weak $L^r$-bounded.
\item Let $\mu\in A_{p/r}\cap {\rm{RH}}_{(s/p)'}$, $\lambda \in A_{q/r}\cap {\rm{RH}}_{(s/q)'}$
 and define the Bloom weight $$\nu^{1+\frac{1}{pm}-\frac{1}{qm}}:= \mu^{\frac{1}{pm}}\lambda^{-\frac{1}{qm}}.$$
\end{itemize}
Then:
\begin{enumerate}[(i)]
\item \label{it:qualitative1} If $p\leq q$ and  $\alpha:= \frac{1}{pm}-\frac1{qm}$, we have
$$
\|T_b^m\|_{L^p(\mu) \to L^q(\la)}\lesssim_{\mu,\lambda} \nrm{b}_{\BMO_{\nu}^{\alpha n}}^m .
$$
\item \label{it:qualitative2}
If $q\leq p$ and  $\frac{1}{t}:=\frac{1}{qm}-\frac1{pm}$ we have
$$
\|T_b^m\|_{L^p(\mu)\to L^q(\la)}\lesssim_{\mu,\lambda} \nrm{M^{\#}_\nu b}_{L^t(\nu)}^m .
$$
\end{enumerate}
\end{theorem}

\begin{proof} Fix $f,g \in L^\infty_c(\R^n)$. By Remark \ref{dyadversion},  there exist $3^n$ dyadic lattices ${\mathscr D}_j$
and sparse families ${\mathcal S}_j\subset {\mathscr D}_j$ such that
$$\int_{{\mathbb R}^n}|T_b^mf||g|\lesssim \sum_{j=1}^{3^n}\big({\mathcal B}_{{\mathcal S}_j,b,r,s}^m(f,g)+{\mathcal B}_{{\mathcal S}_j,b,s',r'}^m(g,f)\big).$$
Therefore, the claims follow from applying Theorem \ref{pleq} twice, directly and dually. We observe that in order to apply Theorem \ref{pleq} to the dual terms
 ${\mathcal B}_{{\mathcal S}_j,b,s',r'}^m(g,f)$ in the dual spaces, we need the conditions
 \begin{itemize}
   \item $\la^{1-q'}\in A_{q'/s'}$
   \item $\mu^{1-p'}\in A_{p'/s'}\cap {\rm{RH}}_{(r'/p')'}$
 \end{itemize}
 which follow directly from our assumptions on $\mu$ and $\lambda$.
\end{proof}

We refer to \cite[Remark 4.4]{LLO22} for a list of operators satisfying the assumptions, and thus the conclusion, of Theorem \ref{wgeneral}. Note that even unweighted bounds for commutators with some of the operators on that list were previously unknown.

Next, we examine a quantitative form of Theorem \ref{wgeneral} in an important particular case of interest.
\begin{theorem}\label{partcase}
Let $1<p<\infty$  and $m\in {\mathbb N}$. Let $T$ be a sublinear operator and $b \in L^1_{\loc}(\R^n)$. Assume the following conditions:
\begin{itemize}
\item Suppose that for all $1<r <2<s<\infty$, both $T$ and ${\mathcal M}^{\#}_{T,s}$ are  locally weak $L^r$-bounded, and
$$\f_{T,r}(\la_{m,n})+\f_{\mc{M}^{\#}_{T,s},r}(\la_{m,n})\leq \psi(r',s),$$
where $\la_{m,n}>0$ is the constant provided by Theorem \ref{sdp} and $\psi\colon [1,\infty)^2 \to [1,\infty)$ is non-decreasing in both variables.
\item Let $\mu,\lambda\in A_{p}$ and define the Bloom weight
$\nu:= (\frac{\mu}{\lambda})^{\frac{1}{pm}}.$
\end{itemize}
Then
$$\|T_b^m\|_{L^p(\mu) \to L^p(\la)}\lesssim K_{p}(\mu,\lambda) C_{p,\psi}(\mu,\la)\nrm{b}_{\BMO_{\nu}}^m ,$$
where
\begin{align*}
  K_{p}(\mu,\lambda) &\phantom{:}\leq \begin{cases}
    \big([\mu]_{A_p}[\la]_{A_p}\big)^{\frac{m+1}{2}\max(1,\frac{1}{p-1})}, &m\geq 1,\\
    [\mu]_{A_p}^{\frac{m}{p-1}}[\la]_{A_p}^{\max\ha{1,\frac{1}{p-1}}}+
[\mu]_{A_p}^{\max(1,\frac{1}{p-1})}[\la]_{A_p}^m, \qquad &m\geq 2,
  \end{cases}\\
  C_{p,\psi}(\mu,\la)&:=\psi\big(c_{p,n,m}\max([\mu]_{A_p},[\la]_{A_p})^{\frac{1}{p-1}}, c_{p,n,m}\max\ha{[\mu]_{A_p},[\la]_{A_p}}\big).
\end{align*}
\end{theorem}

\begin{proof} Fix $f,g \in L^\infty_c(\R^n)$. By Remark \ref{dyadversion}, for any $1<r <2<s<\infty$ there exist $3^n$ dyadic lattices ${\mathscr D}_j$
and sparse families ${\mathcal S}_j\subset {\mathscr D}_j$ such that
\begin{equation}\label{withphi}
\int_{{\mathbb R}^n}|T_b^mf||g|\lesssim  \psi(r',s)\sum_{j=1}^{3^n}\big({\mathcal B}_{{\mathcal S}_j,b,r,s}^m(f,g)+{\mathcal B}_{{\mathcal S}_j,b,s',r'}^m(g,f)\big).
\end{equation}

Suppose that $1<r<\min\ha{\frac{m+1}{m},p}$ and $\max\ha{p,m+1}<s<\infty$. Then $\lfloor rm\rfloor=m$ and $\lfloor s'm\rfloor=m$, and hence, by Theorem \ref{pleq},
$${\mathcal B}_{{\mathcal S}_j,b,r,s}^m(f,g)\lesssim C_{p,r,s}(\mu,\la)\|b\|_{\BMO_{\nu}}^m
\|f\|_{L^p(\mu)}\|g\|_{L^{p'}(\la^{1-p'})},$$
where either
\begin{equation*}
  \begin{aligned}
C_{p,r,s}(\mu,\la)&:= \bigl([\mu]_{A_{p/r}}^{(r-1)m+\frac{m+1}{2r}}[\la]_{A_{p/r}}^{m-\frac{m+1}{2r}} \bigr)^{\max(\frac{1}{r},\frac{1}{p-r})}\\&\hspace{1cm}\cdot
\hab{[\la]_{A_{p/r}}[\la]_{{\rm{RH}}_{(s/p)'}}}^{\max(\frac{s-1}{s-p},\frac{1}{p-r})}.
\end{aligned}
\end{equation*}
or, if $m\geq 2$, alternatively
\begin{equation*}
  C_{p,r,s}(\mu,\la):= [\mu]_{A_{p/r}}^{\frac{rm}{p-r}}\hab{[\la]_{A_{p/r}} [\lambda]_{{\rm{RH}}_{(s/p)'}}}^{\max(\frac{s-1}{s-p},\frac{1}{p-r})}.
\end{equation*}
Moreover, Theorem \ref{pleq} also yields
$${\mathcal B}_{{\mathcal S}_j,b,s',r'}^m(g,f)\lesssim C_{p',s',r'}(\la^{1-p'},\mu^{1-p'})\|b\|_{\BMO_{\nu}}^m
\|g\|_{L^{p'}(\la^{1-p'})}\|f\|_{L^p(\mu)}.$$

Now, let $c_n>0$ be the constant in Proposition \ref{rhweights} and define
\begin{align*}
  \bar r&=  1+\tfrac{1}{p'\cdot c_n}\cdot \max([\mu]_{A_p},[\la]_{A_p})^{-\frac{1}{p-1}} ,\\
  \bar s&=  p\hab{1+ c_n \cdot\max([\mu]_{A_p},[\la]_{A_p})}.
\end{align*}
Then we have for all $1<r\le \bar r$ and $\bar s\le s<\infty$ that
\begin{align*}
  C_{p,r,s}(\mu,\la)&\lesssim \begin{cases}
    \big([\mu]_{A_p}[\la]_{A_p}\big)^{\frac{m+1}{2}\max(1,\frac{1}{p-1})},\qquad&m \geq 1,\\
    [\mu]_{A_p}^{\frac{m}{p-1}}[\la]_{A_p}^{\max\ha{1,\frac{1}{p-1}}}, \qquad &m\geq 2,
  \end{cases}
  \intertext{and, using \eqref{dualityAp}, also}
  C_{p',s',r'}(\la^{1-p'},\mu^{1-p'})&\lesssim \begin{cases}
    \big([\mu]_{A_p}[\la]_{A_p}\big)^{\frac{m+1}{2}\max(1,\frac{1}{p-1})},\qquad&m\geq 1,\\
[\mu]_{A_p}^{\max(1,\frac{1}{p-1})}[\la]_{A_p}^m, \qquad &m\geq 2.
  \end{cases}
\end{align*}
Therefore, if $r =\frac{1}{2}(\min(\frac{m+1}{m},p,\bar r)+1)$ and $s = 2\max(p,m+1,\bar s)$, combining the above estimates with (\ref{withphi}) completes the proof.
\end{proof}

If $T$ and $\mc{M}^{\#}_{T,s}$ are both locally weak $L^1$-bounded, we can take $\psi$ in Theorem \ref{partcase} constant in the first coordinate, which we record as the following corollary.

\begin{cor}\label{r=1}
Let $1<p<\infty$  and $m\in {\mathbb N}$. Let $T$ be a sublinear operator and $b \in L^1_{\loc}(\R^n)$. Assume the following conditions:
\begin{itemize}
\item Suppose that for all $2<s<\infty$, both $T$ and  ${\mathcal M}^{\#}_{T,s}$ are  locally weak $L^1$-bounded, and
$$\f_{\mc{M}^{\#}_{T,s},1}(\la_{m,n})\leq \psi(s),$$
where $\la_{m,n}$ is a constant provided by Theorem \ref{sdp} and $\psi\colon [1,\infty) \to [1,\infty)$ is non-decreasing.
\item Let $\mu,\lambda\in A_{p}$ and define the Bloom weight
$\nu:= (\frac{\mu}{\lambda})^{\frac{1}{pm}}.$
\end{itemize}
Then
$$\|T_b^m\|_{L^p(\mu) \to L^p(\la)}\lesssim K_p(\mu,\lambda) C_{p,\psi}(\mu,\la)\nrm{b}_{\BMO_{\nu}}^m ,$$
where $K_p(\mu,\lambda)$ is as in Theorem \ref{partcase} and
$$C_{p,\psi}(\mu,\la):=\psi\big(c_{p,n,m}\max([\mu]_{A_p},[\la]_{A_p})\big).$$
\end{cor}

\begin{proof}
Since $T$ and ${\mathcal M}^{\#}_{T,s}$ are  locally weak $L^1$-bounded, they are locally weak $L^r$ bounded for all $r>1$, and
\begin{align*}
\f_{T,r}(\la_{m,n})+\f_{\mc{M}^{\#}_{T,s},r}(\la_{m,n})&\leq\f_{T,1}(\la_{m,n})+\f_{\mc{M}^{\#}_{T,s},1}(\la_{m,n})\\
&\le \f_{T,1}(\la_{m,n})+\psi(s).
\end{align*}
Applying Theorem \ref{partcase} with
$$\psi(r',s):=\f_{T,1}(\la_{m,n})+\psi(s)$$
finishes the proof.
\end{proof}

\subsection{Calder\'on--Zygmund operators} \label{subs:CZO} As discussed in the introduction, Bloom weighted estimates for commutators have been widely studied for Calder\'on--Zygmund operators. As a first application of our results, we will compare the weighted estimates that we obtained in the diagonal $p=q$ case for Calder\'on--Zygmmund operators and discuss their sharpness in the sense of Definition \ref{sharp}. In particular, let us prove Theorem \ref{shm1}.

Recall that a linear operator $T$ is called Dini-continuous Calder\'on-Zygmund operator if it is $L^2$-bounded
and for $f \in L^\infty_c(\R^n)$ has a representation
$$Tf(x)=\int_{{\mathbb R}^n}K(x,y)f(y)dy\qquad x\not\in \text{supp}\,f,$$
where
$$|K(x,y)-K(x',y)|+|K(y,x)-K(y,x')|\le \o\left(\frac{|x-x'|}{|x-y|}\right)\frac{1}{|x-y|^n},$$
whenever $|x-y|>2|x-x'|$, with the modulus of continuity $\o:[0,1]\to [0,\infty)$ satisfying $\int_0^1\o(t)\frac{dt}{t}<\infty.$

\begin{proof}[Proof of Theorem \ref{shm1}]
It is well-known that $T$ and ${\mathcal M}^{\#}_{T,\infty}$ are of weak $L^1$-bounded (see, e.g., \cite{Gr14b, LO20}).
Therefore, after normalizing such that $\|b\|_{{\BMO}_{\nu}}=1$, by Corollary \ref{r=1} we have
\begin{equation}\label{prevv}
\|T_b^m\|_{L^p(\mu)\to L^p(\la)}\lesssim ([\la]_{A_p}[\mu]_{A_p})^{\frac{m+1}{2}\max(1,\frac{1}{p-1})}.
\end{equation}
and, if $m\geq 2$, also
\begin{equation}\label{kangweiv}
\|T_b^m\|_{L^p(\mu)\to L^p(\la)}\lesssim [\la]_{A_p}^{\max(1,\frac{1}{p-1})}[\mu]_{A_p}^{\frac{m}{p-1}}+[\mu]_{A_p}^{\max(1,\frac{1}{p-1})}[\la]_{A_p}^{m}.
\end{equation}

First let us take $m\geq 2$.  Then (\ref{kangweiv}) implies for $p\geq 2$ that
$$\|T_b^m\|_{L^p(\mu)\to L^p(\la)}\lesssim [\mu]_{A_p}^{\max(1,\frac{m}{p-1})}[\la]_{A_p}^{m}.$$
If $p>\frac{1+3m}{m+1}$, then $\max(1,\frac{m}{p-1})<\frac{m+1}{2}$, and therefore we obtain that \eqref{prevv} is not sharp in the sense of Definition \ref{sharp}. Moreover, for $1<p\leq 2$  \eqref{kangweiv} implies
$$\|T_b^m\|_{L^p(\mu)\to L^p(\la)}\lesssim [\mu]_{A_p}^{\frac{m}{p-1}}[\la]_{A_p}^{\max(m,\frac{1}{p-1})}.$$
If $p<\frac{1+3m}{2m}$, then $\max(m,\frac{1}{p-1})<\frac{m+1}{2}\frac{1}{p-1}$. Therefore, we again obtain that (\ref{prevv}) is not sharp.

\medskip

Suppose now that $m=1$. Let us show that in this case (\ref{prevv}) is sharp for all $p \in (1,\infty)$.
By duality, the sharpness of the exponent $\max(1,\frac{1}{p-1})$ of $[\la]_{A_p}$ is equivalent to the sharpness of the exponent $\max(1,\frac{1}{p'-1})$ of $[\mu]_{A_{p'}}$.
Therefore, it suffices to establish the sharpness of the exponent $\frac{1}{p-1}$ of $[\la]_{A_p}$ and $[\mu]_{A_p}$ for $1<p\le 2$.

Let $1<p\le 2$. We will provide examples in dimension $n=1$ for the Hilbert transform $H$.
Let us prove first that the exponent $\frac{1}{p-1}$ of $[\la]_{A_p}$ cannot be decreased. Let $0<\d<1$ and define $\mu:=1$ and $\lambda(x):=|x|^{(p-1)(1-\delta)}$. It is well known that $[\la]_{A_p}\eqsim \delta^{1-p}$.
Observe that $\nu=\lambda^{-1/p}=|x|^{\frac{(\delta-1)}{p'}}$. Define $b:=\nu$. It is easy to see that $\|b\|_{\BMO_{\nu}}\le 2$.
Define $f(x):=|x|^{\frac{(\delta-1)}{p}} \chi_{(0,1)}(x)$. Then
$$([\la]_{A_p}[\mu]_{A_p})^{\frac{1}{p-1}}\|f\|_{L^p(\mu)}\eqsim \tfrac{1}{\d^{1+1/p}}.$$
Therefore, the sharpness of the exponent $\frac{1}{p-1}$ would follow if we show that
\begin{equation}\label{leftside}
\|H_b^1\|_{L^p(\la)}\gtrsim \tfrac{1}{\d^{1+1/p}}.
\end{equation}

Observe that for $x \in \R$
\begin{align*}
H_{b}^1f(x)&=|x|^{\frac{(\delta-1)}{p'}} H(|y|^{\frac{(\delta-1)}{p}}\chi_{(0,1)})(x)- H(|y|^{\delta-1}\chi_{(0,1)})(x)\\
&=:h_1(x)-h_2(x).
\end{align*}
By the unweighted $L^p$ boundedness of $H$,
$$\|h_1\|_{L^p(\la)}=\Big(\int_{{\mathbb R}}|H(|y|^{\frac{(\delta-1)}{p}}\chi_{(0,1)})(x)|^pdx\Big)^{1/p}\lesssim \tfrac{1}{\d^{1/p}}.$$
On the other hand, $\frac{1}{\delta} |x|^{\delta-1}\lesssim |h_2(x)|$ for all $x\in(0,1)$, and therefore
$$\tfrac{1}{\delta^{1+1/p}} \lesssim \|h_2\|_{L^p(\lambda)},$$
which proves (\ref{leftside}).

Let us show now that the exponent $\frac{1}{p-1}$ of $[\mu]_{A_p}$ cannot be decreased. The example is very similar.
Define $\la:=1$ and $\mu(x):=|x|^{(p-1)(1-\delta)}$. Then $[\mu]_{A_p}\eqsim \delta^{1-p}$.
Observe that $\nu=\mu^{1/p}=|x|^{\frac{(1-\delta)}{p'}}$. Define $b:=\nu$, for which we have $\|b\|_{\BMO_{\nu}}\eqsim 1$.
Define $f(x):=|x|^{\delta-1} \chi_{(0,1)}(x)$. Then
$$([\la]_{A_p}[\mu]_{A_p})^{\frac{1}{p-1}}\|f\|_{L^p(\mu)}\eqsim \tfrac{1}{\d^{1+1/p}}.$$
Therefore, the sharpness of the exponent $\frac{1}{p-1}$ would follow if we show that
\begin{equation}\label{leftside1}
\|H_b^1\|_{L^p}\gtrsim \tfrac{1}{\d^{1+1/p}}.
\end{equation}

Observe that
\begin{align*}
H_{b}^1f(x)&=|x|^{\frac{(1-\delta)}{p'}}H(|y|^{(\delta-1)}\chi_{(0,1)})(x)-H(|y|^{\frac{\delta-1}{p}}\chi_{(0,1)})(x)\\
&=h_1(x)-h_2(x).
\end{align*}
Exactly as above, $\|h_2\|_{L^p}\lesssim \tfrac{1}{\d^{1/p}}$. On the other hand, $\frac{1}{\d}|x|^{\frac{\d-1}{p}}\lesssim h_1(x)$ for all $x\in (0,1)$, and therefore,
$$\frac{1}{\delta^{1+1/p}} \lesssim \|h_1\|_{L^p(\lambda)},$$
which proves (\ref{leftside1}). This completes the proof.
\end{proof}

As we mentioned in the introduction, Theorem \ref{shm1} leaves open a question about the sharpness of  (\ref{lor19}) when $m\ge 2$ and  $p\in [\frac{1+3m}{2m},\frac{1+3m}{m+1}]$.
Indeed, our example is based on the obvious fact that $\nu\in\BMO_{\nu}$. In the case $m=1$ the choice $b=\nu$ shows the sharpness of (\ref{lor19}) for all $1<p<\infty$.
However, it is easy to check that in the case $m\ge 2$ the same choice is not enough in order to show the sharpness of (\ref{lor19}).

\subsection{Further applications}
We conclude this paper by applying our results to several other concrete examples of operators, for which (quantitative) Bloom weighted bounds of their commutators have not been known before.

Given an operator $T$ and $s\geq 1$, define the (non-sharp) grand maximal truncation operator ${\mathcal M}_{T,s}$ by
$${\mathcal M}_{T,s}f(x):=\sup_{Q\ni x}\Big(\frac{1}{|Q|}\int_Q|T(f\chi_{{\mathbb R}^n\setminus 3Q})|^s\Big)^{1/s}, \qquad x \in \R^n.$$
Observe that
\begin{equation}\label{mp}
{\mathcal M}^{\#}_{T,s}f\le 2\, {\mathcal M}_{T,s}f.
\end{equation}

\begin{example}\label{rsio}
Consider a class of rough homogeneous singular integrals defined by
$$T_{\Omega}f(x)=\text{p.v.}\int_{{\mathbb R}^n}f(x-y)\frac{\O(y/|y|)}{|y|^n}\dd y, \qquad x \in \R^d,$$
for $\O\in L^{\infty}(S^{n-1})$ with zero average over the sphere. Fix $p \in (1,\infty)$, let $m\in \N$ and $b \in L^1_{\loc}(\R^n)$.

It is a well-known result of Seeger \cite{Se96} that $T_{\Omega}$ is of weak type $(1,1)$. Moreover, it was shown by the first author  \cite{Le19} that for $s>2$, the grand maximal truncation operator ${\mathcal M}_{T_{\O},s}$ is weak $L^1$-bounded with
$$\|{\mathcal M}_{T_{\O},s}\|_{L^{1}({\mathbb R}^n)\to L^{1,\infty}({\mathbb R}^n)}\lesssim s.$$
Therefore, by (\ref{mp}), ${\mathcal M}_{T_{\O},s}^{\#}$  satisfies the same bound. From this, by Corollary \ref{r=1}, we obtain for $\mu,\lambda \in A_p$ and $\nu:= (\frac{\mu}{\lambda})^{\frac{1}{pm}}$
$$\|(T_{\O})_b^m\|_{L^p(\mu) \to L^p(\la)}\lesssim C_1(\mu,\la)\nrm{b}_{\BMO_{\nu}}^m,$$
where
$$C_1(\la,\mu):=\max\hab{[\mu]_{A_p},[\la]_{A_p}}K_p(\mu,\lambda)$$
with $K_p(\mu,\lambda)$ as in Theorem \ref{partcase}.

Furthermore, since $T_{\Omega}$ is essentially self-adjoint, we have
$$\|(T_{\O})_b^m\|_{L^p(\mu)\to L^p(\la)}=\|(T_{\O})_b^m\|_{L^{p'}(\la^{1-p'})\to L^{p'}(\mu^{1-p'})}.$$
Using \eqref{dualityAp}, this yields
$$\|(T_{\O})_b^m\|_{L^p(\mu)\to L^p(\la)}\lesssim C_2(\mu,\la)\nrm{b}_{\BMO_{\nu}}^m,$$
where
\begin{align*}
C_2(\la,\mu)&:=\max\hab{[\la^{1-p'}]_{A_{p'}},[\mu^{1-p'}]_{A_{p'}}}K_{p'}(\lambda^{1-p'},\mu^{1-p'})\\
&\phantom{:}=\max\hab{[\mu]_{A_p},[\la]_{A_p}}^{\frac{1}{p-1}}K_p(\mu,\lambda).
\end{align*}
Therefore, we finally obtain that
\begin{align*}
    \|(T_{\O})_b^m\|_{L^p(\mu) \to L^p(\la)}&\lesssim \max\hab{[\mu]_{A_p},[\la]_{A_p}}^{\min(1,\frac{1}{p-1})} K_p(\mu,\lambda)\,\nrm{b}_{\BMO_{\nu}}^m.
\end{align*}
\end{example}

\begin{example}\label{mrsio}
Consider now a class of maximal rough homogeneous singular integrals defined by
$$T_{\Omega}^{\star}f(x)=\sup_{\e>0}\Big|\int_{|y|>\e}f(x-y)\frac{\O(y/|y|)}{|y|^n}dy\Big|, \qquad x \in \R^d,$$
for $\O\in L^{\infty}(S^{n-1})$ with zero average over the sphere. Fix $p \in (1,\infty)$, let $m\in \N$ and $b \in L^1_{\loc}(\R^n)$.

It was shown by Di Plinio--Hyt\"onen--Li \cite{DPHL20} that for $1<r<2$,
\begin{equation}\label{weaktstar}
\|T_{\O}^{\star}f\|_{L^{r,\infty}(\R^n)}\lesssim r'\|f\|_{L^r(\R^n)}.
\end{equation}
Let us deduce from the recent work of Tao--Hu \cite{TH23} that
\begin{equation}\label{locweak}
\f_{\mc{M}^{\#}_{T_{\O}^{\star},s},r}(\la_{m,n})\lesssim s\log r'.
\end{equation}

Denote $\Phi(t):=t\log\log({\rm e}^2+t)$. In \cite{TH23} the authors established that for $s >2$ and $\alpha>0$,
$$
|\{x\in {\mathbb R}^n: {\mathcal M}_{T_{\O}^*,s}f(x)>\a\}|\lesssim s\int_{{\mathbb R}^n}\Phi\left(\frac{|f|}{\a}\right)dx.
$$
From this, using the estimate
$$\Phi(t)\lesssim (\log r')t+t^r,\qquad t>0$$
for $1<r<2$, we obtain
$$
|\{x\in Q: {\mathcal M}_{T_{\O}^*,s}(f\chi_Q)(x)>\a\}|\lesssim s\Big((\log r')\int_Q\frac{|f|}{\a}+\int_Q\Big(\frac{|f|}{\a}\Big)^r\Big).
$$
Hence, for $\lambda \in (0,1)$, we have
$$
\absb{\cbraceb{x\in Q: {\mathcal M}_{T_{\O}^*,s}(f\chi_Q)(x)>\tfrac{s\log r'}{\la}\langle|f|\rangle_{r,Q}}}\lesssim \la|Q|.
$$
Along with (\ref{mp}), this implies (\ref{locweak}).

Using (\ref{weaktstar}) and (\ref{locweak}), we are in position to apply Theorem~\ref{partcase} with $\psi(r',s):=r'+(\log r')s$, from which it follows that for $\mu,\lambda \in A_p$ and $\nu:= (\frac{\mu}{\lambda})^{\frac{1}{pm}}$
$$\|(T_{\O}^{\star})_b^m\|_{L^p(\mu) \to L^p(\la)}\lesssim \hab{t^{\frac{1}{p-1}}+t\log t } K_p(\mu,\lambda)\nrm{b}_{\BMO_{\nu}}^m $$
with $K_p(\mu,\lambda)$ as in Theorem~\ref{partcase} and $t = \max\hab{[\mu]_{A_p},[\la]_{A_p}}$
\end{example}

\begin{example}\label{br}
Consider the Bochner-Riesz operator at the critical index $B_{(n-1)/2}$, which is defined by
$$
(B_{(n-1)/2}f)\,\widehat\,\,(\xi):=(1-|\xi|^2)_{+}^{\frac{n-1}{2}}\widehat f(\xi), \qquad \xi \in \R^n.
$$
It is a well-known result of Christ \cite{Ch88} that $B_{(n-1)/2}$ is of weak $L^1$-bounded. Furthermore, it is implicit in the work of Shrivastava--Shuin \cite{SS20,SS21} that for $s \in [1,\infty)$
\begin{equation}\label{brs}
\|{\mathcal M}_{B_{(n-1)/2},s}\|_{L^{1}({\mathbb R}^n)\to L^{1,\infty}({\mathbb R}^n)}\lesssim s.
\end{equation}
Therefore, arguing as in Example \ref{rsio}, we have for $p \in (1,\infty)$, $m\in \N$ and $b \in L^1_{\loc}(\R^n)$ that for all $\mu,\lambda \in A_p$ and $\nu:= (\frac{\mu}{\lambda})^{\frac{1}{pm}}$
\begin{align*}
    \|(B_{(n-1)/2})_b^m\|_{L^p(\mu) \to L^p(\la)}&\lesssim\max\hab{[\mu]_{A_p},[\la]_{A_p}}^{\min(1,\frac{1}{p-1})} C_p(\mu,\lambda)\,\nrm{b}_{\BMO_{\nu}}^m.
\end{align*}
with $C_p(\mu,\lambda)$ as in Theorem \ref{partcase}.

We add some details about (\ref{brs}). It is well-known (see, e.g., \cite[Section 5.2.1]{Gr14b}) that the kernel of $B_{(n-1)/2}$ is given by
$K_{(n-1)/2}:=K+\Phi$, where
$$K(x):=c_n\frac{\cos(2\pi|x|-\pi n/2)}{|x|^n}\chi_{\{|x|\ge 1\}}, \qquad x \in \R^n,$$
and $|\Phi(x)|\lesssim \frac{1}{1+|x|^{n+1}}$. Next, let $\phi$ be a radial smooth function supported in $B(0,2)$ such that $\phi=1$ on $B(0,1)$ and set $$\psi_j(x):=\phi(2^{-j}|x|)-\phi(2^{-j+1}|x|), \qquad x \in \R^n.$$ Then we obtain that
$$B_{(n-1)/2}f(x)=T_1f(x)+T_2f(x)+Tf(x),$$
where
\begin{align*}
  T_1f(x)&:=f*\Phi(x),&&x \in \R^n,\\ T_2f(x)&:=f*(\phi K)(x), &&x \in \R^n,\\
  Tf(x)&:=\sum_{j=1}^{\infty}f*(\psi_jK)(x), &&x \in \R^n.
\end{align*}

Using that $|T_if|\lesssim Mf$ for $i=1,2$, we obtain
$${\mathcal M}_{B_{(n-1)/2},s}f(x)\lesssim Mf(x)+{\mathcal M}_{T,s}f(x), \qquad x \in \R^n.$$
Therefore, it suffices to prove (\ref{brs}) for ${\mathcal M}_{T,s}$. In order to do that, we fix an integer $N$ and further decompose $T=S_1+S_2$, where
$$S_1f(x):=\sum_{j=1}^Nf*(\psi_jK)(x),\qquad  S_2f(x):=\sum_{j=N+1}^{\infty}f*(\psi_jK)(x).$$
Let $\e:=2^{-N}$. It was shown in \cite{SS21} that $S_1$ is a Dini-continuous Calder\'on-Zygmund operator with Dini-constant bounded by $\log\frac{1}{\e}$, and that $\|S_2\|_{L^2\to L^2}\lesssim \e^{\a}$
for some $\a\in (0,1]$. Moreover, the kernels $\psi_jK$ have radial smoothness (as was observed in \cite{LPRR19}), which makes it possible to apply Seeger's machinery from \cite{Se96}.
Using all these ingredients, the proof goes through as in \cite{Le19} by the first author.
\end{example}

\begin{remark}
Let  $p,q \in (1,\infty)$, $m\in \N$ and $b \in L^1_{\loc}(\R^n)$.
  By a similar argument as employed in Subsection \ref{subs:CZO} and Examples \ref{rsio}, \ref{mrsio} and \ref{br}, now using Theorem~\ref{wgeneral}, we also get for all $\mu \in A_p$ and $\lambda \in A_q$ and $$\nu^{1+\frac{1}{pm}-\frac{1}{qm}}:= \mu^{\frac{1}{pm}}\lambda^{-\frac{1}{qm}}$$
  that we have
\begin{align*}
    \|T_b^m\|_{L^p(\mu) \to L^q(\la)}&\lesssim_{\mu,\lambda} \begin{cases}
      \nrm{b}_{\BMO^{\alpha n}_{\nu}}^m, \qquad  \qquad & p\leq q,\\
      \nrm{M^{\#}_\nu (b)}_{L^t(\nu)}^m, & q \leq p.
    \end{cases}
\end{align*}
with $\alpha := \frac{1}{pm}-\frac1{qm}$,  $\frac{1}{t} := \frac{1}{qm}-\frac1{pm}$ and $T$ is either a Calder\'on--Zygmund operator or  $T \in \cbrace{ T_\Omega, T_\Omega^*, B_{(n-1)/2}}$.
We leave the exact dependence on $[\mu]_{A_p}$ and $[\lambda]_{A_q}$ in these cases to the interested reader.
\end{remark}

\bibliographystyle{plain}
\bibliography{commutatorbib}

\end{document}